\documentclass{lms}

\usepackage{amsmath,mathtools}
\usepackage{amssymb,amsfonts}
\usepackage{microtype}
\usepackage[unicode,bookmarksnumbered,hidelinks,hypertexnames=false]{hyperref}

\newtheorem{theorem}{Theorem}[section]

\newtheorem{lemma}[theorem]{Lemma}
\newtheorem{corollary}[theorem]{Corollary}
\newnumbered{definition}[theorem]{Definition}
\newnumbered{remark}[theorem]{Remark}

\newenvironment{romanlist}
  {\begin{enumerate}[iii]\renewcommand{\theenumi}{\textup{(\roman{enumi})}}}
  {\end{enumerate}}
\newenvironment{alphlist}
  {\begin{enumerate}[b]\renewcommand{\theenumi}{\textup{(\alph{enumi})}}}
  {\end{enumerate}}

\DeclareMathOperator{\Spec}{Spec}
\DeclareMathOperator{\Proj}{Proj}
\DeclareMathOperator{\Span}{span}

\DeclareMathOperator{\Hilb}{Hilb}
\DeclareMathOperator{\edim}{edim}

\DeclareMathOperator{\im}{im}
\DeclareMathOperator{\ch}{ch}
\DeclareMathOperator{\td}{td}

\newcommand{\F}{\mathbb F}
\newcommand{\PP}{\mathbb P}
\newcommand{\Ample}{\mathcal A}
\newcommand{\OO}{\mathcal O}

\newcommand{\LL}{\mathcal L}
\newcommand{\MM}{\mathcal M}
\newcommand{\EE}{\mathcal E}
\newcommand{\FF}{\mathcal F}
\newcommand{\kk}{\kappa}
\newcommand{\mm}{\mathfrak m}
\newcommand{\barX}{\overline{X}}
\newcommand{\set}[1]{\left\{#1\right\}}

\makeatletter
\def\@footnotetwo{\footnotetext{2020 {\normalfont\itshape Mathematics Subject Classification\/} \@classno.}}
\makeatother
\journal{\vbox to 5.5pt{\noindent
  \parbox[t]{\textwidth}{\raggedleft\normalfont\footnotesize\baselineskip 9pt
  {\itshape Submitted exclusively to the Transactions of the London Mathematical Society}}
  \vss}}

\title[Anti-Bertini embeddings over finite fields]
 {Hyperplane anti-Bertini embeddings over finite fields}

\author{Yutong Zhang     Yaoran Yang}

\classno{14G15 (primary), 14N05, 14C20, 14M10 (secondary)}

\begin{document}
\maketitle

\begin{abstract}
Baker asked, as recorded by Poonen, whether a fixed smooth quasiprojective variety over a finite field must have a smooth rational hyperplane section after every sufficiently high-dimensional linearly nondegenerate embedding.  Poonen predicted a negative answer for every positive-dimensional variety.  We prove this predicted negative answer for each prescribed variety: if $X$ is nonempty, smooth, quasiprojective, and of pure positive dimension over $\F_q$, then for every sufficiently large $N$ there is a locally closed embedding $X\hookrightarrow\PP^N_{\F_q}$ whose components remain linearly nondegenerate after arbitrary scalar extension, but whose every $\F_q$-rational hyperplane section is singular.  The construction assigns one closed point of $X$ to each rational hyperplane and forces the pulled-back linear form to have zero first-order jet at that point.
\end{abstract}

\begingroup
\makeatletter
\renewcommand\@pnumwidth{2em}
\renewcommand\@tocrmarg{3em}
\renewcommand\@dotsep{4.5}
\renewcommand*\l@section[2]{%
  \@dottedtocline{1}{0pt}{1.9em}{#1}{#2}%
}
\renewcommand*\l@subsection[2]{%
  \@dottedtocline{2}{1.9em}{2.8em}{#1}{#2}%
}
\makeatother
\setcounter{tocdepth}{2}
\pdfbookmark[1]{Contents}{toc}
\tableofcontents
\endgroup

\section{Introduction}

Let $k=\F_q$.  The classical Bertini smoothness theorem over an infinite field says that, if
\[
        X\subset \PP^n_k
\]
is smooth and quasiprojective of pure dimension $m\ge1$, then there is a dense Zariski-open subset of the dual projective space whose points parametrize hyperplanes $H\subset\PP^n_k$ for which
\[
        H\cap X\quad\text{is smooth of dimension}\quad m-1.
\]
Over a finite field this dense open subset may have no $k$-rational point.  Poonen's Bertini theorem over finite fields resolves this difficulty in a different direction: for a fixed embedding, one lets the hypersurface degree tend to infinity.  Without additional Taylor conditions, the density of smooth sections is
\[
        \zeta_X(m+1)^{-1}
        =\prod_{P\in |X|}\left(1-q^{-(m+1)\deg P}\right).
\]
More generally, in the notation of Poonen's theorem, if Taylor data are imposed on a finite subscheme $Z$, if $T\subset H^0(Z,\OO_Z)$ is the set of allowed restrictions, and if
\[
        U=X-(Z\cap X),
\]
then the corresponding density is
\[
        \frac{\#T}{\#H^0(Z,\OO_Z)}\,\zeta_U(m+1)^{-1};
\]
see \cite[Theorems 1.1 and 1.2]{Poonen2004}.

Poonen also recorded the following question of Baker.  Fix a smooth quasiprojective $X$ of pure dimension $m$ over $\F_q$.  Is there an integer $n_0$ such that, for every $n\ge n_0$ and every embedding
\[
        \iota:X\longrightarrow \PP^n_k
\]
whose connected components are not mapped into hyperplanes, some $k$-rational hyperplane $H\subset\PP^n_k$ has
\[
        H\cap \iota(X)\quad\text{smooth of dimension}\quad m-1?
\]
Poonen conjectured that the answer should be negative for every positive-dimensional $X$ \cite[Question 4.1]{Poonen2004}.  Since Baker's question quantifies over all embeddings, a negative answer for a given $X$ means constructing bad embeddings of that same $X$ in arbitrarily large projective spaces.

Erman and Wood gave such a negative answer in an existential form: for every $m>0$ they constructed a smooth projective $m$-fold with a sequence of nondegenerate embeddings
\[
        \kappa_d:X\hookrightarrow \PP^{N_d}_k,
        \qquad N_d\to\infty,
\]
for which every $k$-rational hyperplane section is singular \cite[Proposition 9.3]{ErmanWood2015}.  Their examples choose a special variety $X$.  They also noted that a different construction, suggested by Poonen, should work for any smooth quasiprojective $X$, although not by using complete linear series \cite[Remark 9.4]{ErmanWood2015}.  The present paper gives that construction for the prescribed variety $X$ appearing in Baker's question.

\subsection{Main result and scope}

The theorem below gives a negative answer to Baker's question for each nonempty smooth quasiprojective variety of positive dimension over a finite field.  It is not a statement that all high-dimensional embeddings are bad.  Rather, for the given $X$, it constructs linearly nondegenerate embeddings with no smooth rational hyperplane section.  The construction in fact works in every sufficiently large ambient dimension; this is a precise output of the method, not a separate conjectural point.

\begin{theorem}[Anti-Bertini embeddings for a prescribed variety]\label{thm:main}
Let $X$ be a nonempty smooth quasiprojective $k$-scheme of pure dimension $m\ge 1$.  Then there exists an integer $N_X$ with the following property.  For every $N\ge N_X$ there is a locally closed $k$-embedding
\[
        \iota_N:X\hookrightarrow \PP^N_k
\]
such that
\begin{romanlist}
\item for every connected component $C$ of $X$, every field extension $K/k$, and every connected component $D$ of $C_K$, the image of $D$ under $(\iota_N)_K$ is not contained in any hyperplane of $\PP^N_K$;
\item for every $k$-rational hyperplane $H\subset \PP^N_k$, the scheme-theoretic intersection
\[
        H\cap \iota_N(X)
\]
is singular at a closed point of $\iota_N(X)$; in particular it is not smooth of dimension $m-1$ over $k$.
\end{romanlist}
More precisely, one may assign to every $H\in(\PP^N_k)^\vee(k)$ a distinct closed point $P_H\in X$ such that, if $s_H\in H^0(X,\iota_N^*\OO_{\PP^N}(1))$ is the pullback of a defining linear form for $H$, then
\[
        s_H\big|_{P_H^{(2)}}=0,
        \qquad
        P_H^{(2)}:=\Spec\bigl(\OO_{X,P_H}/\mm_{P_H}^2\bigr),
\]
and $s_H$ is not identically zero on the connected component of $X$ containing $P_H$; consequently $H\cap\iota_N(X)$ is singular at $\iota_N(P_H)$.
\end{theorem}

Thus, for the variety $X$ itself, there is no threshold after which every nondegenerate embedding has a smooth rational hyperplane section.  The theorem also makes clear what is not being asserted: the embeddings are constructed sublinear systems, and the paper does not claim that complete linear series have the same behaviour for arbitrary $X$.

\subsection{Proof strategy and organization}

Fix $N$ and write
\[
        \Lambda_N=(\PP^N_k)^\vee(k).
\]
This is a finite set.  We choose pairwise distinct closed points
\[
        S_N=\set{P_\lambda\mid \lambda\in\Lambda_N}\subset X
\]
and split the $N+1$ projective coordinates into two blocks,
\[
        N+1=A+B,
        \qquad
        A\ge m+3,
        \qquad
        B=(m+1)(r+1).
\]
The first block, denoted
\[
        u_1,\ldots,u_A,
\]
is the anti-code block.  For a hyperplane $\lambda=[a:b]$, its $a$-part imposes one linear relation on the first-order jets of the $u_\ell$ at the assigned point $P_\lambda$.  The local construction keeps that one relation while still making the first-order neighbourhood $P_\lambda^{(2)}$ embed into $\PP^{A-1}$; the inequality $A\ge m+3$ supplies the needed room.

The second block is a shielded Segre block.  Starting from a fixed componentwise nondegenerate embedding defined by sections $e_0,\ldots,e_r$, we construct sections $h_0,\ldots,h_m$ that vanish on every $P_\lambda^{(2)}$ but have no common zero away from the assigned points.  After replacing the $h_i$ by high powers, the sections
\[
        w_{ij}=h_i^n e_j
\]
have zero first-order jet at all assigned points, so they do not affect the imposed singularities, while away from those points they provide the separation needed for an embedding.

For a rational hyperplane $\lambda=[a:b]$, the pulled-back linear form is
\[
        s_\lambda=
        \sum_{\ell=1}^{A}a_\ell u_\ell+
        \sum_{i=0}^{m}\sum_{j=0}^{r}b_{ij}w_{ij}.
\]
By construction $s_\lambda$ vanishes on $P_\lambda^{(2)}$.  Since the image of the connected component containing $P_\lambda$ is not contained in the hyperplane, this section is not identically zero there.  In a regular local ring, a nonzero local equation lying in $\mm^2$ defines a singular hypersurface.  Hence the hyperplane section is singular at $P_\lambda$.

The rest of the paper supplies the details.  Section~\ref{sec:conventions} records finite interpolation and base-change criteria.  Section~\ref{sec:local-shield} proves the local anti-code lemma and constructs the shield.  Section~\ref{sec:blocks} fixes the two coordinate blocks for a chosen $N$.  Section~\ref{sec:anticode} prescribes the global anti-code block by interpolation.  Section~\ref{sec:embedding} proves that the resulting map is an embedding and that every rational hyperplane section is singular.  Section~\ref{sec:comparison} states the consequence for Baker's question and compares the scope with earlier constructions.

The argument is independent of the density method in \cite{Poonen2004} and of semiample Bertini probabilities in \cite{ErmanWood2015}.  Recent extensions of Poonen's theorem vary the hypersurface degree in a fixed embedding rather than varying the embedding to defeat all rational hyperplanes; see \cite{Poonen2008,BucurKedlaya2012,BiluHowe2021,Bertucci2025,GhoshKrishna2023}.  Earlier finite-field pathologies and related transversality or irreducibility refinements include Katz's space-filling constructions and correction \cite{Katz1999,Katz2001}, the Charles--Poonen Bertini irreducibility theorem \cite{CharlesPoonen2016}, and results on transverse lines, pencils, and linear subspaces \cite{Asgarli2019,AsgarliGhioca2022,AsgarliDuanLai2024}.  We use the Lang--Weil estimate for closed points \cite{LangWeil1954} and standard facts on very ample systems, projection, and positivity \cite{Hartshorne1977,Jouanolou1983,Lazarsfeld2004}.

\section{Conventions and interpolation}\label{sec:conventions}

Throughout the paper
\[
        k=\F_q,
        \qquad
        |Y|=\text{the set of closed points of }Y,
        \qquad
        \kk(P)=\OO_{Y,P}/\mm_P.
\]
The hyperplanes forced to be singular are the $k$-rational hyperplanes; the nondegeneracy conclusion below is stated after arbitrary scalar extension.  If $X/k$ is smooth of pure dimension $m$ and $P\in |X|$, then
\[
        P^{(2)}=\Spec\bigl(\OO_{X,P}/\mm_P^2\bigr),
        \qquad
        \dim_{\kk(P)}H^0(P^{(2)},\OO_{P^{(2)}})=m+1.
\]
For a line bundle $\LL$ on $X$, a choice of trivialization over $P^{(2)}$ identifies
\[
        H^0(P^{(2)},\LL|_{P^{(2)}})
        \simeq \OO_{X,P}/\mm_P^2.
\]
The condition $s|_{P^{(2)}}=0$ is independent of this trivialization.  When a nonzero value of a section of a line bundle is prescribed at a finite reduced point, a trivialization of the relevant fibre is understood; changing the trivialization multiplies all prescribed nonzero values by a unit and does not affect the vanishing assertions used below.

We use Grothendieck's one-dimensional quotient convention
\[
        \PP_S(\mathcal F)=\Proj_S\bigl(\operatorname{Sym}_{\OO_S}(\mathcal F)\bigr).
\]
Thus \(\PP_S(\mathcal F)\) parametrizes invertible quotients of \(\mathcal F\), and
\(\OO_{\PP_S(\mathcal F)}(1)\) is the tautological quotient line bundle.  In particular, a basepoint-free linear system
\(W\subset H^0(Y,\mathcal L)\) defines a morphism \(Y\to\PP(W)\).  For a finite-dimensional vector space \(E\), we write
\[
        \PP_{\mathrm{lin}}(E)=\PP(E^\vee)
\]
for the projective space of one-dimensional subspaces of \(E\).

We use two forms of finite interpolation.  The first is the usual projective statement.

\begin{lemma}[Projective finite interpolation]\label{lem:interpolation-projective}
Let $\barX$ be a projective $k$-scheme, let $\Ample$ be an ample line bundle on $\barX$, let $\FF$ be a coherent sheaf on $\barX$, and let $Z\subset \barX$ be a finite closed subscheme.  For all $d\gg0$ the restriction map
\[
        H^0(\barX,\FF\otimes\Ample^{\otimes d})
        \longrightarrow
        H^0(Z,(\FF\otimes\Ample^{\otimes d})|_Z)
\]
is surjective.
\end{lemma}

\begin{proof}
Let
\[
        \mathcal K=
        \ker\bigl(\FF\longrightarrow i_*i^*\FF\bigr),
        \qquad i:Z\hookrightarrow\barX.
\]
Then
\[
0\longrightarrow \mathcal K
\longrightarrow \FF
\longrightarrow i_*i^*\FF
\longrightarrow0
\]
is exact.  Tensoring by the locally free sheaf $\Ample^{\otimes d}$ preserves exactness, and using the projection formula for the finite morphism $i$ gives
\[
0\longrightarrow \mathcal K\otimes\Ample^{\otimes d}
\longrightarrow \FF\otimes\Ample^{\otimes d}
\longrightarrow i_*i^*(\FF\otimes\Ample^{\otimes d})
\longrightarrow0 .
\]
The connecting homomorphism for global sections has target
\[
        H^1(\barX,\mathcal K\otimes\Ample^{\otimes d}),
\]
which is zero for all $d\gg0$ by Serre vanishing, because $\mathcal K$ is coherent.  Hence the displayed restriction map is surjective.
\end{proof}

\begin{lemma}[Quasiprojective finite interpolation]\label{lem:interpolation-quasi}
Let $X$ be a quasiprojective $k$-scheme, let $j:X\hookrightarrow \barX$ be a projective closure with $X$ open in $\barX$, and let $\Ample$ be an ample line bundle on $\barX$.  Let $\EE$ be a line bundle on $X$ and let $Z\subset X$ be a finite closed subscheme.  Then, for all $d\gg0$, the restriction map
\[
        H^0(X,\EE\otimes j^*\Ample^{\otimes d})
        \longrightarrow
        H^0(Z,(\EE\otimes j^*\Ample^{\otimes d})|_Z)
\]
is surjective.
\end{lemma}

\begin{proof}
Since $\barX$ is noetherian and $X\subset\barX$ is open, the coherent sheaf $\EE$ extends to a coherent sheaf $\FF$ on $\barX$ with $\FF|_X\simeq\EE$.  The composite morphism
\[
        Z\hookrightarrow X\hookrightarrow\barX
\]
is a closed immersion: indeed, $Z$ is finite over $k$, hence proper over $k$, and its map to the separated $k$-scheme $\barX$ is proper; since $Z\to X\to\barX$ is also a locally closed immersion, it is a proper immersion and therefore a closed immersion.  Apply Lemma~\ref{lem:interpolation-projective} to $\FF$ and to this closed subscheme $Z\subset\barX$.  The resulting global sections on $\barX$ restrict to sections on $X$ with the prescribed values on $Z$.
\end{proof}

We shall repeatedly require closed points avoiding finite sets.

\begin{lemma}[Closed points of arbitrarily large degree]\label{lem:closedpoints}
Let $Y$ be a finite type $k$-scheme containing a positive-dimensional irreducible component, and let $F\subset |Y|$ be finite.  Then $Y\setminus F$ contains closed points of arbitrarily large degree over $k$.
\end{lemma}

\begin{proof}
Choose a positive-dimensional irreducible component of $Y$, give it the reduced induced structure, and choose a nonempty affine normal open subscheme
\[
        Y_0\subset Y\setminus F
\]
inside that component.  Let $k_0$ be the algebraic closure of $k$ in $k(Y_0)$.  Since $Y_0$ is normal, every element of $k_0$ is regular on $Y_0$: on each nonempty affine open $\Spec A\subset Y_0$, it is integral over $k\subset A$ and therefore belongs to the normal domain $A$.  Thus the structure morphism of $Y_0$ factors through $\Spec k_0$.  The field $k_0$ is perfect and algebraically closed in $k(Y_0)$, so the integral $k_0$-scheme $Y_0$ is geometrically integral over $k_0$.

Write $q_0=\#k_0$ and $d=\dim Y_0\ge1$.  Lang--Weil gives, for the degree-$n$ extension $k_{0,n}/k_0$,
\[
        \#Y_0(k_{0,n})=q_0^{dn}+O(q_0^{(d-1/2)n}).
\]
Take $n$ through sufficiently large primes.  The only proper intermediate field between $k_0$ and $k_{0,n}$ is then $k_0$.  Consequently every $k_{0,n}$-point that is not defined over $k_0$ has exact degree $n$ over $k_0$.  Since $Y_0(k_0)$ is finite, Lang--Weil shows that such points exist for arbitrarily large prime $n$.  They correspond to closed points of $Y_0$, hence of $Y\setminus F$, whose degrees over $k$ are $n[k_0:k]$.  These degrees tend to infinity.
\end{proof}

\begin{lemma}[Connected smooth components and constants]\label{lem:component-constants}
Let $C$ be a nonempty smooth connected finite type $k$-scheme.  Then $C$ is integral.  If $k_C$ denotes the algebraic closure of $k$ in the function field $k(C)$, then $k_C/k$ is finite separable, the structure morphism of $C$ factors through $\Spec k_C$, and $C$ is geometrically integral over $k_C$.
\end{lemma}

\begin{proof}
The field $k$ is perfect.  Since $C/k$ is smooth, $C$ is regular.  Every local ring of a regular scheme is a regular local ring, hence a domain; therefore a regular scheme is locally irreducible and locally reduced.  It follows that distinct irreducible components of the noetherian scheme $C$ cannot meet.  The irreducible components are consequently pairwise disjoint open-and-closed subsets.  Since $C$ is nonempty and connected, it has a single irreducible component.  Thus $C$ is irreducible and reduced, hence integral.

The extension $k(C)/k$ is finitely generated, so the algebraic closure $k_C$ of the finite field $k$ inside $k(C)$ is a finite field.  Hence $k_C/k$ is finite separable.  Let $\alpha\in k_C$.  On any nonempty affine open $U=\Spec A\subset C$, the element $\alpha\in k(C)=\operatorname{Frac}(A)$ satisfies a monic polynomial with coefficients in $k\subset A$.  Since $C$ is regular, $A$ is a normal domain, so $\alpha\in A$.  Thus every element of $k_C$ is a global regular function on $C$, and the structure morphism factors through $\Spec k_C$.

Since $k_C/k$ is finite and $C$ is of finite type over $k$, the induced $k_C$-scheme $C$ is of finite type over $k_C$.  Finally, $k_C$ is perfect and is algebraically closed in $k(C)$ by definition.  The standard criterion for geometric integrality over a perfect field then implies that the integral $k_C$-scheme $C$ is geometrically integral over $k_C$.
\end{proof}

\begin{lemma}[Base change for global sections on quasi-compact quasi-separated schemes]\label{lem:H0-base-change}
Let $F$ be a field, let $Y$ be a quasi-compact quasi-separated $F$-scheme, and let $\mathcal L$ be a line bundle on $Y$.  For every field extension $E/F$, the natural map
\[
        H^0(Y,\mathcal L)\otimes_F E
        \longrightarrow
        H^0(Y_E,\mathcal L_E)
\]
is an isomorphism.  In particular, a finite set of $F$-linearly independent global sections remains linearly independent after base change to $E$.
\end{lemma}

\begin{proof}
Choose a finite affine open cover $Y=\bigcup_{i=1}^n U_i$.  Since $Y$ is quasi-separated, every intersection $U_i\cap U_j$ is quasi-compact; choose a finite affine open cover
\[
        U_i\cap U_j=\bigcup_a U_{ij,a}.
\]
For the quasi-coherent sheaf $\mathcal L$, the sheaf equalizer for this finite affine cover is exact:
\[
0\longrightarrow H^0(Y,\mathcal L)
\longrightarrow \prod_i H^0(U_i,\mathcal L)
\rightrightarrows
\prod_{i,j,a}H^0(U_{ij,a},\mathcal L).
\]
Tensoring with the flat $F$-module $E$ preserves this equalizer and commutes with the finite products.  On every affine open, sections of a quasi-coherent sheaf commute with extension of scalars.  After base change, the affine opens $(U_i)_E$ and $(U_{ij,a})_E$ give the corresponding equalizer computing $H^0(Y_E,\mathcal L_E)$.  This proves the isomorphism.  The assertion about linear independence is immediate.
\end{proof}

\begin{lemma}[Componentwise independence after arbitrary scalar extension]\label{lem:component-base-change}
Let $C$ be as in Lemma~\ref{lem:component-constants}, let $\mathcal L$ be a line bundle on $C$, and let
\[
        s_0,\ldots,s_r\in H^0(C,\mathcal L)
\]
be $k_C$-linearly independent.  For every field extension $K/k$ and every connected component $D$ of $C_K$, the restrictions $s_0|_D,\ldots,s_r|_D$ are linearly independent.  If, in addition, $s_0,\ldots,s_r$ are basepoint-free on $C$, then the morphism defined by these sections is defined on $D$ and does not map $D$ into any hyperplane.
\end{lemma}

\begin{proof}
Because $k_C/k$ is finite separable, the $K$-algebra $k_C\otimes_k K$ is finite etale and decomposes as a product of finite separable field extensions
\[
        k_C\otimes_k K\simeq \prod_\alpha K_\alpha .
\]
Since $C$ is geometrically integral over $k_C$,
\[
        C_K\simeq C\times_{k_C}(k_C\otimes_k K)
        \simeq \coprod_\alpha C_{K_\alpha},
\]
and the schemes $C_{K_\alpha}$ are precisely the connected components of $C_K$.  The $k_C$-scheme $C$ is of finite type, hence quasi-compact and quasi-separated, so Lemma~\ref{lem:H0-base-change} gives
\[
        H^0(C,\mathcal L)\otimes_{k_C}K_\alpha
        \simeq
        H^0(C_{K_\alpha},\mathcal L_{K_\alpha}).
\]
Tensoring a finite linearly independent set with a field preserves linear independence, proving the first assertion.  If $s_0,\ldots,s_r$ are basepoint-free on $C$, their base changes define the base-changed morphism on each $C_{K_\alpha}$.  After viewing the component $C_{K_\alpha}$ as a $K_\alpha$-scheme, containment of its image in a hyperplane over $K_\alpha$ would give a nontrivial $K_\alpha$-linear relation among the restricted sections, contradicting the linear independence just proved.  In particular, containment in a hyperplane over the smaller field $K$ is impossible.
\end{proof}

\begin{lemma}[Projective length-two closed-immersion criterion]\label{lem:linear-system-criterion}
Let $Y$ be a projective $k$-scheme, let $\mathcal L$ be a line bundle on $Y$, and let
$W\subset H^0(Y,\mathcal L)$ be a finite-dimensional basepoint-free $k$-linear system.  Suppose that, for every algebraically closed field $\Omega/k$ and every zero-dimensional closed subscheme
\[
        \xi\subset Y_\Omega
\]
of length two, the restriction map
\[
        W_\Omega\longrightarrow H^0(\xi,\mathcal L_\Omega|_\xi)
\]
is surjective.  Then the morphism
\[
        \phi_W:Y\longrightarrow \PP(W)
\]
associated with $W$ is a closed immersion.
\end{lemma}

\begin{proof}
We first prove the two geometric consequences of the length-two hypothesis.  Let $\Omega/k$ be algebraically closed.  If $y_1,y_2\in Y(\Omega)$ are distinct, then for the reduced length-two subscheme
\[
        \xi=\{y_1\}\amalg\{y_2\}
\]
the map
\[
        W_\Omega\to \mathcal L_{y_1}\oplus\mathcal L_{y_2}
\]
is surjective.  Hence there is a section vanishing at $y_1$ and not at $y_2$, so the two evaluation quotients of $W_\Omega$ are not proportional.  Thus $\phi_{W,\Omega}(y_1)\ne\phi_{W,\Omega}(y_2)$.

Next let $0\ne v\in T_yY_\Omega$ be a tangent vector at a geometric point $y$.  The corresponding morphism
\[
        \Spec \Omega[\varepsilon]/(\varepsilon^2)\longrightarrow Y_\Omega
\]
is a closed immersion onto a length-two subscheme $\xi_v$, because the induced nonzero map
\[
        \mathfrak m_y/\mathfrak m_y^2\longrightarrow (\varepsilon)
\]
is surjective.  Choose $s_0\in W_\Omega$ with $s_0(y)\ne0$.  On the affine chart of $\PP(W_\Omega)$ defined by $s_0\ne0$, the morphism is given by the functions $s/s_0$ with $s\in W_\Omega$.  If $d\phi_{W,\Omega}(v)=0$, all these functions have zero first-order variation along $\xi_v$.  Equivalently, the image of
\[
        W_\Omega\to H^0(\xi_v,\mathcal L_\Omega|_{\xi_v})
\]
is generated by the restriction of $s_0$ and is therefore at most one-dimensional.  This contradicts the assumed surjectivity, since the target has dimension two over $\Omega$.  Hence every geometric tangent map is injective.

We now pass to the morphism over $k$.  Let $y\in Y$ and choose an algebraic closure $\Omega$ of $\kk(y)$.  The injectivity of the tangent map at the corresponding geometric point is dual to the surjectivity of
\[
        \phi_W^*\Omega_{\PP(W)/k}\otimes\Omega
        \longrightarrow
        \Omega_{Y/k}\otimes\Omega .
\]
By right exactness of the standard cotangent sequence
\[
        \phi_W^*\Omega_{\PP(W)/k}
        \longrightarrow
        \Omega_{Y/k}
        \longrightarrow
        \Omega_{Y/\PP(W)}
        \longrightarrow0,
\]
we get
\[
        \Omega_{Y/\PP(W)}\otimes\Omega=0.
\]
Since field extension from $\kk(y)$ to $\Omega$ is faithfully flat, this implies
\[
        \Omega_{Y/\PP(W)}\otimes\kk(y)=0.
\]
Nakayama's lemma then shows that the coherent sheaf $\Omega_{Y/\PP(W)}$ vanishes at $y$.  As $y$ was arbitrary, $\Omega_{Y/\PP(W)}=0$.  The morphism $\phi_W$ is locally of finite type, hence it is unramified.

For an unramified morphism the diagonal is an open immersion.  We claim that the diagonal
\[
        Y\longrightarrow Y\times_{\PP(W)}Y
\]
is surjective on points.  A point of the fibre product with residue field $F$ gives two $F$-valued points of $Y$ with the same image in $\PP(W)$.  After extending $F$ to an algebraic closure $\Omega$, these become two geometric points of $Y_\Omega$ with the same image.  The geometric-point injectivity proved above forces them to be equal after base change to $\Omega$; faithful flat descent for morphisms from $\Spec F$ then shows that the two original $F$-valued points were equal.  Hence the original point of the fibre product lies in the image of the diagonal.  Thus the diagonal is a surjective open immersion, hence an isomorphism, and $\phi_W$ is a monomorphism.

Since $Y$ is projective over $k$ and $\PP(W)$ is separated over $k$, the morphism $\phi_W$ is proper.  A proper monomorphism of schemes is a closed immersion \cite[Tag 04XV]{StacksProject}.  Therefore $\phi_W$ is a closed immersion.
\end{proof}

\begin{lemma}[Length-two span locus and projection]\label{lem:length-two-projection}
Let $Y\subset\PP(V)$ be a smooth locally closed $k$-subscheme of pure dimension $m$, embedded by a finite-dimensional linear system $V$.  Let $\Sigma_2(Y)$ be the Zariski closure of the union of the spans of all geometric length-two subschemes of $Y$.  Then
\[
        \dim \Sigma_2(Y)\le 2m+1.
\]
If $W\subset V$ is a subspace and $\PP(V/W)\cap\Sigma_2(Y)=\varnothing$, then $W$ separates every geometric length-two subscheme of $Y$.  Moreover, if $W$ fails to separate some geometric length-two subscheme, then $\PP(V/W)$ meets the span of that subscheme and hence meets $\Sigma_2(Y)$.
\end{lemma}

\begin{proof}
After base change to $\bar k$, the Hilbert scheme $\Hilb^2(Y_{\bar k})$ has dimension at most $2m$.  The locus of two distinct reduced points is dominated by $Y_{\bar k}\times Y_{\bar k}\setminus\Delta$ and has dimension at most $2m$.  If $m=0$, the smooth scheme $Y_{\bar k}$ is reduced and has no nonreduced length-two subschemes.  If $m\ge1$, the nonreduced locus is parametrized by a point of $Y_{\bar k}$ together with a line in its tangent space, and has dimension $m+(m-1)\le2m-1$.  Dimension is unchanged by extension of the ground field, so
\[
        \dim\Hilb^2(Y)\le 2m.
\]
For each length-two subscheme $\xi$, the image of $\xi$ in $\PP(V)$ spans a projective linear space of dimension at most one.  Therefore the union of these spans is dominated by a projective bundle of relative dimension at most one over $\Hilb^2(Y)$, and has dimension at most $2m+1$.

For the projection assertion, fix an algebraically closed extension $\Omega/k$ and a length-two subscheme $\xi\subset Y_\Omega$.  The composite
\[
        \xi\longrightarrow Y_\Omega\longrightarrow \PP(V_\Omega)
\]
is a locally closed immersion.  Since $\xi$ is finite over $\Omega$, it is proper over $\Omega$; because projective space is separated over $\Omega$, the composite is proper.  A proper locally closed immersion is a closed immersion, so we regard $\xi$ as a closed length-two subscheme of $\PP(V_\Omega)$.

Linear forms on projective space restrict surjectively to every closed subscheme of length two.  Indeed, choose a linear form $\ell_0$ that is nonzero at every point of the support of $\xi$.  On the affine chart $D_+(\ell_0)$, the functions $\ell/\ell_0$, with $\ell$ ranging over ambient linear forms, generate $H^0(\xi,\OO_\xi)$ as an $\Omega$-algebra because $\xi\hookrightarrow D_+(\ell_0)$ is a closed immersion.  This algebra has $\Omega$-dimension two, so the unit together with one of these affine linear functions spans it.  After using $\ell_0$ to trivialize $\OO(1)|_\xi$, it follows that ambient linear forms, identified with $V_\Omega$, restrict surjectively to
\[
        H^0(\xi,\OO_{\PP(V_\Omega)}(1)|_\xi).
\]
Equivalently, the scheme-theoretic projective span of $\xi$ is a line.
If $W_\Omega$ does not restrict surjectively, then some one-dimensional quotient of this two-dimensional vector space annihilates the image of $W_\Omega$.  Composing with the surjection from $V_\Omega$ gives a point of the projective span of $\xi$ whose quotient map kills $W_\Omega$, and hence factors through $V_\Omega/W_\Omega$.  Thus this point lies in $\PP((V/W)_\Omega)$.  Conversely, a point of the span lying in $\PP((V/W)_\Omega)$ gives a one-dimensional quotient annihilating the image of $W_\Omega$, so $W$ fails to separate $\xi$.  This proves both assertions about projection and separation.
\end{proof}

\begin{lemma}[A finite-field total-degree bound for rational points]\label{lem:finite-field-degree-bound}
For a projective $k$-scheme $B\subset\PP^n_k$, define its total projective degree by
\[
        \deg^{\mathrm{tot}} B
        =\sum_{Z\in\operatorname{Irr}(B_{\mathrm{red}})}\deg Z,
\]
where every irreducible component of $B_{\mathrm{red}}$, in every dimension, is included in the sum.  There is a constant $C_q$, depending only on $q$, with the following property.  If $\dim B\le D$, then, for every $e\ge1$,
\[
        \#B(\F_{q^e})\le C_q\deg^{\mathrm{tot}}(B)q^{De}.
\]
The same bound holds with any number $\Delta\ge\deg^{\mathrm{tot}}(B)$ in place of $\deg^{\mathrm{tot}}(B)$.
\end{lemma}

\begin{proof}
The set of rational points is unchanged after passing to $B_{\mathrm{red}}$.  Put $Q=q^e$ and
\[
        B_e=(B_{\F_Q})_{\mathrm{red}}\subset\PP^n_{\F_Q}.
\]
Then $B(\F_Q)=B_e(\F_Q)$.  Since $k$ is perfect and $\F_Q/k$ is finite separable, the base change of the reduced finite-type $k$-scheme $B_{\mathrm{red}}$ is reduced; hence
\[
        B_e=(B_{\mathrm{red}})_{\F_Q}.
\]
Let $Z$ be an irreducible component of $B_{\mathrm{red}}$, and let
\[
        Z_{\F_Q}=Z'_1\cup\cdots\cup Z'_a
\]
be its decomposition into irreducible components.  All $Z'_\nu$ have dimension $\dim Z$.  The Hilbert polynomial of $Z_{\F_Q}\subset\PP^n_{\F_Q}$ is obtained from that of $Z\subset\PP^n_k$ by extension of scalars and therefore has the same leading coefficient.  Since the intersections of distinct $Z'_\nu$ have smaller dimension, the leading coefficient of the Hilbert polynomial of their union is the sum of their leading coefficients.  Hence
\[
        \sum_{\nu=1}^a\deg Z'_\nu=\deg Z.
\]
Every irreducible component of $Z_{\F_Q}$ dominates $Z$ under the finite projection $Z_{\F_Q}\to Z$.  Therefore a component lying over $Z$ cannot coincide with a component lying over a distinct irreducible component of $B_{\mathrm{red}}$.  Summing the displayed equality over all irreducible components of $B_{\mathrm{red}}$ gives
\[
        \deg^{\mathrm{tot}} B_e
        =\deg^{\mathrm{tot}} B.
\]

Put
\[
        \pi_a=1+Q+\cdots+Q^a\quad(a\ge0),
        \qquad
        \pi_a=0\quad(a<0).
\]
Let $W\subset B_e$ be an irreducible component of dimension $d_W\le D$.  If $d_W=n$, then $W=\PP^n_{\F_Q}$, so
\[
        \#W(\F_Q)=\pi_n=\deg(W)\pi_{d_W}.
\]
If $d_W<n$, Couvreur's arbitrary-projective-variety bound \cite{Couvreur2015} over the field $\F_Q$ gives
\[
        \#W(\F_Q)
        \le
        \deg(W)(\pi_{d_W}-\pi_{2d_W-n})+\pi_{2d_W-n}.
\]
Since $\deg(W)\ge1$, this again implies
\[
        \#W(\F_Q)
        \le
        \deg(W)\pi_{d_W}
        =
        \deg(W)(1+Q+\cdots+Q^{d_W}).
\]
Thus the last inequality holds for every irreducible component $W$ of $B_e$.
The $\F_Q$-points of $B_e$ are contained in the union of the $\F_Q$-points of its irreducible components, so
\[
\begin{aligned}
        \#B(\F_Q)
        &=\#B_e(\F_Q)                                                     \\
        &\le
        \sum_{W\in\operatorname{Irr}(B_e)}
        \deg(W)(1+Q+\cdots+Q^{d_W})                                      \\
        &\le
        (1-q^{-1})^{-1}q^{De}
        \sum_{W\in\operatorname{Irr}(B_e)}\deg(W)                         \\
        &=
        (1-q^{-1})^{-1}\deg^{\mathrm{tot}}(B)q^{De}.
\end{aligned}
\]
Thus one may take $C_q=(1-q^{-1})^{-1}$.  Replacing $\deg^{\mathrm{tot}}(B)$ by any larger $\Delta$ only enlarges the right-hand side.
\end{proof}

\begin{lemma}[Polynomial total-degree bound for the compactified length-two parameter space]\label{lem:length-two-degree}
Let $\barX$ be a projective $k$-scheme, let $\Ample$ be a very ample line bundle on $\barX$, and assume that
\[
        \mathcal H=\Hilb^2(\barX)
\]
is nonempty.  Put
\[
        V_d=H^0(\barX,\Ample^{\otimes d}).
\]
Let
\[
        p:\mathcal U\to\mathcal H,
        \qquad
        q:\mathcal U\to\barX
\]
be the universal family and the evaluation morphism.  For $d\ge1$, set
\[
        \mathcal E_d=p_*q^*\Ample^{\otimes d}.
\]
Then $\mathcal E_d$ is locally free of rank two, the natural map $V_d\otimes_k\OO_{\mathcal H}\to\mathcal E_d$ is surjective, and the projective bundle
\[
        B_d=\PP_{\mathcal H}(\mathcal E_d)
\]
of one-dimensional quotients, with projection $\pi_d:B_d\to\mathcal H$ and tautological quotient $\mathcal Q_d=\OO_{B_d}(1)$, satisfies
\[
        \dim B_d\le b_{\barX}:=\dim\mathcal H+1.
\]
Fix a very ample line bundle $\OO_{\mathcal H}(1)$ on $\mathcal H$ and put
\[
        \mathcal R_d=\mathcal Q_d\otimes \pi_d^*\OO_{\mathcal H}(1).
\]
Then $\mathcal R_d$ is very ample on $B_d$.  Moreover, for every fixed integer $0\le n\le b_{\barX}$ and every monomial $M$ of total degree $n$ in the two divisor classes
\[
        c_1(\mathcal Q_d),\qquad \pi_d^*c_1(\OO_{\mathcal H}(1)),
\]
the sum
\[
        \sum_{\substack{Z\in\operatorname{Irr}((B_d)_{\mathrm{red}})\\ \dim Z=n}}
        \left|\int_Z M\right|
\]
is bounded, for $d\gg0$, by a polynomial in $d$.  In particular, if
\[
        \deg^{\mathrm{tot}}_{\mathcal R_d}B_d
        =\sum_{Z\in\operatorname{Irr}((B_d)_{\mathrm{red}})}
        \int_Z c_1(\mathcal R_d)^{\dim Z},
\]
then there is a polynomial $P(d)$ such that
\[
        \deg^{\mathrm{tot}}_{\mathcal R_d}B_d\le P(d)
\]
for all $d\gg0$.
\end{lemma}

\begin{proof}
The universal family $p:\mathcal U\to\mathcal H$ is finite flat of degree two.  Since $p$ is finite locally free of rank two and $q^*\Ample^{\otimes d}$ is invertible on $\mathcal U$,
\[
        \mathcal E_d=p_*q^*\Ample^{\otimes d}
\]
is locally free of rank two.  Since $\Ample^{\otimes d}$ is very ample for $d\ge1$, its global sections separate every length-two subscheme of $\barX$; equivalently,
\[
        V_d\otimes\OO_{\mathcal H}\longrightarrow \mathcal E_d
\]
is surjective.  The dimension bound follows from the fact that $B_d$ is a projective bundle of relative dimension one over $\mathcal H$.

The surjection $V_d\otimes\OO_{\mathcal H}\twoheadrightarrow\mathcal E_d$ realizes $B_d$ as a closed subscheme of $\mathcal H\times\PP(V_d)$, where both projective bundles are taken in the one-dimensional quotient convention.  Under this closed immersion, $\mathcal Q_d$ is the restriction of $\OO_{\PP(V_d)}(1)$ and $\pi_d^*\OO_{\mathcal H}(1)$ is the restriction of $\OO_{\mathcal H}(1)$.  Therefore $\mathcal R_d$ is the restriction of $\OO(1,1)$ and is very ample.

It remains to prove the polynomial cyclewise intersection bounds.  We first record why the relative Riemann--Roch formula used below applies after base change to an irreducible component of $\mathcal H_{\mathrm{red}}$.  Let $R$ be a local ring on an affine open of $\mathcal H$, and put $A_R=(p_*\OO_{\mathcal U})_R$.  This is a finite free $R$-algebra of rank two.  The unit $1\in A_R$ is unimodular: for every prime $\mathfrak p\subset R$, its image in $A_R\otimes_R\kappa(\mathfrak p)$ is the nonzero unit.  After shrinking the affine open, choose $t\in A_R$ such that $1,t$ is an $R$-basis.  Then
\[
        t^2=\alpha+\beta t
        \qquad(\alpha,\beta\in R),
\]
and the natural map
\[
        R[T]/(T^2-\beta T-\alpha)\longrightarrow A_R,
        \qquad T\longmapsto t,
\]
is an isomorphism, because both sides are free with bases $1,T$ and $1,t$, respectively.  The monic polynomial $T^2-\beta T-\alpha$ is a non-zero-divisor in $R[T]$.  Thus $p$ is locally the zero scheme of one non-zero-divisor in the affine line over the base, hence is a finite local complete intersection morphism, including in characteristic two.

Fix an irreducible component $W\subset\mathcal H_{\mathrm{red}}$.  We shall prove the required bounds for the corresponding component of $B_d$ over $W$.  By de Jong's alteration theorem there are an alteration
\[
        g:\widetilde W\longrightarrow W
\]
and a dense open immersion of $\widetilde W$ into a regular projective integral $k$-variety.  Since $W$ is projective and $g$ is proper, $\widetilde W$ is proper over $k$; the open immersion is therefore proper, hence also closed, and consequently is an isomorphism onto the integral compactification.  Thus $\widetilde W$ is regular, projective, and integral.  Because the finite field $k$ is perfect, $\widetilde W$ is smooth over $k$.  Write $\delta=[k(\widetilde W):k(W)]>0$ for the generic degree of $g$; see \cite[Theorem 4.1 and \S4.7]{deJong1996}.

Set
\[
        \widetilde{\mathcal U}=\mathcal U\times_{\mathcal H}\widetilde W,
        \qquad
        \widetilde p:\widetilde{\mathcal U}\to\widetilde W,
        \qquad
        \widetilde q:\widetilde{\mathcal U}\to\barX .
\]
Finite flat base change gives
\[
        g^*(\mathcal E_d|_W)\simeq \widetilde p_*\widetilde q^*\Ample^{\otimes d}.
\]
The morphism $\widetilde p$ is again finite flat of degree two and local complete intersection.  Since $\widetilde W$ is smooth over $k$, the composite $\widetilde{\mathcal U}\to\Spec k$ is local complete intersection, and in $K^0(\widetilde{\mathcal U})$ its virtual tangent bundle is
\[
        T_{\widetilde{\mathcal U}/k}
        =T_{\widetilde p}+\widetilde p^*T_{\widetilde W/k},
\]
where $T_{\widetilde p}$ is the virtual relative tangent bundle.  Since $\widetilde p$ is finite, its higher direct images vanish, so $\widetilde p_!=\widetilde p_*$ in coherent $K$-theory.  Apply the covariance and local-complete-intersection normalization of the singular Riemann--Roch transformation \cite[Theorem 18.3]{Fulton1998}.  On the smooth target $\widetilde W$ one has
\[
        \tau_{\widetilde W}(\mathcal F)
        =\ch(\mathcal F)\,\td(T_{\widetilde W})\cap[\widetilde W],
\]
while on $\widetilde{\mathcal U}$ one has
\[
        \tau_{\widetilde{\mathcal U}}(\widetilde q^*\Ample^{\otimes d})
        =\exp(d\widetilde q^*a)\,\td(T_{\widetilde p})\,
        \widetilde p^*\td(T_{\widetilde W})\cap[\widetilde{\mathcal U}],
        \qquad a=c_1(\Ample).
\]
Covariance and the projection formula give
\[
\begin{aligned}
        &\ch(g^*(\mathcal E_d|_W))\,\td(T_{\widetilde W})\cap[\widetilde W]\\
        &\qquad=
        \td(T_{\widetilde W})\cap
        \widetilde p_*
        \bigl(\exp(d\widetilde q^*a)\,\td(T_{\widetilde p})\cap[\widetilde{\mathcal U}]\bigr).
\end{aligned}
\]
The degree-zero term of $\td(T_{\widetilde W})$ is $1$, so cap product by this Todd class is an automorphism of $A_*(\widetilde W)_{\mathbb Q}$.  Cancelling it yields
\[
        \ch(g^*(\mathcal E_d|_W))\cap[\widetilde W]
        =
        \widetilde p_*
        \bigl(\exp(d\widetilde q^*a)\,\td(T_{\widetilde p})\cap[\widetilde{\mathcal U}]\bigr).
\]
In every fixed codimension, only finitely many powers of $\widetilde q^*a$ contribute.  Moreover, because $\widetilde W$ is smooth, cap product with $[\widetilde W]$ identifies $A^i(\widetilde W)_{\mathbb Q}$ with $A_{\dim\widetilde W-i}(\widetilde W)_{\mathbb Q}$.  Taking the codimension-$i$ component of the displayed equality therefore shows that each class
\[
        \ch_i(g^*(\mathcal E_d|_W))\in A^i(\widetilde W)_{\mathbb Q}
\]
is a polynomial in $d$ with coefficients in the fixed group $A^i(\widetilde W)_{\mathbb Q}$.  The Chern classes are universal polynomials with rational coefficients in the Chern-character components.  It follows that every product, of bounded total codimension, of Chern classes and of Segre classes of either $g^*(\mathcal E_d|_W)$ or its dual is polynomial in $d$ as a Chow class on $\widetilde W$.

Write
\[
        \xi=c_1(\mathcal Q_d),
        \qquad
        h=c_1(\OO_{\mathcal H}(1)).
\]
The irreducible components of $(B_d)_{\mathrm{red}}$ are the inverse images of the irreducible components of $\mathcal H_{\mathrm{red}}$.  Hence their number and dimensions are independent of $d$.  Let
\[
        C_d=\PP_W(\mathcal E_d|_W)
\]
be the corresponding component of $(B_d)_{\mathrm{red}}$.  We use the convention
\[
        s(\mathcal F)=c(\mathcal F)^{-1}.
\]
For the projective bundle of one-dimensional quotients $\pi:\PP(\mathcal F)\to W$, with $\xi=c_1(\OO_{\PP(\mathcal F)}(1))$, the projective bundle formula is
\[
        \pi_*(\xi^{\operatorname{rank}\mathcal F-1+j})=s_j(\mathcal F^\vee).
\]
Consequently, for a monomial $\xi^a\pi_d^*h^{n-a}$ of total degree $n=\dim C_d$,
\[
        \int_{C_d}\xi^a\pi_d^*h^{n-a}
        =
        \int_W h^{n-a}\,s_{a-1}((\mathcal E_d|_W)^\vee)\cap[W],
\]
with the convention $s_{-1}=0$.  By the projection formula and $g_*[\widetilde W]=\delta[W]$,
\[
\begin{aligned}
        &\int_W h^{n-a}\,s_{a-1}((\mathcal E_d|_W)^\vee)\cap[W] \\
        &\qquad=
        \frac1\delta
        \int_{\widetilde W}g^*h^{n-a}\,
        s_{a-1}((g^*(\mathcal E_d|_W))^\vee)\cap[\widetilde W].
\end{aligned}
\]
The right-hand side is a rational polynomial function of $d$ by the preceding paragraph.  Consequently each monomial intersection number on $C_d$ is a rational polynomial function of $d$, and its absolute value is bounded, for $d\gg0$, by a real polynomial in $d$.

Only finitely many components $W$ and finitely many monomials of total degree at most $b_{\barX}$ occur.  Thus the sum of the absolute values of the monomial intersection numbers is bounded by a single polynomial in $d$.  Finally,
\[
        c_1(\mathcal R_d)^n=(\xi+\pi_d^*h)^n.
\]
Expanding this binomial for each possible component dimension $n\le b_{\barX}$ and using the just-proved monomial bounds gives
\[
        \deg^{\mathrm{tot}}_{\mathcal R_d}B_d\le P(d)
\]
for some polynomial $P(d)$ and all sufficiently large $d$.
\end{proof}

\begin{lemma}[Polynomial bad-locus bound for compactified length-two projection]\label{lem:length-two-bad-locus}
Use the notation of Lemma~\ref{lem:length-two-degree}, put
\[
        M_d=\dim_k V_d,
        \qquad
        b=b_{\barX}.
\]
For $L>b$, let $\mathcal B_{d,L}\subset V_d^L$ be the set of ordered $L$-tuples $(s_1,\ldots,s_L)$ for which some geometric length-two subscheme of $\barX$ is not separated by the span of the $s_i$.  Then there are an integer $d_0$, independent of $L$, and a polynomial $P_0(d,L)$ such that, for all $d\ge d_0$ and all $L>b$,
\[
        \#\mathcal B_{d,L}(k)
        \le
        P_0(d,L)q^{LM_d-(L-b)}.
\]
\end{lemma}

\begin{proof}
Let $\pi_d:B_d\to\mathcal H$ be as in Lemma~\ref{lem:length-two-degree}.  Pulling the universal quotient to $B_d$ gives
\[
        V_d\otimes_k\OO_{B_d}\twoheadrightarrow\mathcal Q_d,
\]
and let $\mathcal K_d$ be its kernel.  Thus $\mathcal K_d$ is locally free of rank $M_d-1$.  Since $\mathcal H$ is nonempty and $V_d\otimes_k\OO_{\mathcal H}\twoheadrightarrow\mathcal E_d$ is a quotient of rank two, one has $M_d\ge2$.  Hence $\mathcal K_d$ and all the projective bundles used below have positive rank.

We first spell out the length-two quotient criterion used below.  Let $\Omega/k$ be algebraically closed and let $\xi\subset \barX_\Omega$ be a length-two subscheme.  Since $V_d$ separates every length-two subscheme of $\barX$, the map
\[
        V_d\otimes_k\Omega\longrightarrow H^0(\xi,\Ample^{\otimes d}|_\xi)
\]
is surjective.  An ordered tuple $(s_1,\ldots,s_L)$ fails to separate $\xi$ precisely when the span of the images of the $s_i$ is a proper subspace of this two-dimensional target.  This is equivalent to the existence of a one-dimensional quotient of $H^0(\xi,\Ample^{\otimes d}|_\xi)$ on which all $s_i$ vanish.  Such a quotient is exactly a geometric point of $B_d$ over $\xi$, and the condition that all $s_i$ vanish in it says that the tuple lies in the fibre of $\mathcal K_d^{\oplus L}$ at that point.

We use $\PP(\mathcal F)$ for the projective bundle of one-dimensional quotients of a vector bundle $\mathcal F$, and write
\[
        \PP_{\mathrm{lin}}(E)=\PP(E^\vee)
\]
for the projective space of one-dimensional subspaces of a vector space $E$.  Put
\[
        I_{d,L}=\PP_{B_d}\bigl((\mathcal K_d^{\oplus L})^\vee\bigr).
\]
It parametrizes a point of $B_d$ together with a one-dimensional subspace of the fibre of $\mathcal K_d^{\oplus L}$, that is, a nonzero ordered $L$-tuple of sections, up to common scalar, whose entries all vanish in the corresponding one-dimensional quotient of a length-two jet.  The inclusion $\mathcal K_d^{\oplus L}\subset V_d^{\oplus L}\otimes\OO_{B_d}$ induces a projective morphism
\[
        f:I_{d,L}\longrightarrow \PP_{\mathrm{lin}}(V_d^{\oplus L}).
\]
Let $Y_{d,L}$ be the scheme-theoretic image.  Let $v\in\mathcal B_{d,L}(k)$ be nonzero, and let $[v]\in\PP_{\mathrm{lin}}(V_d^{\oplus L})(k)$ be its projectivization.  The quotient criterion over an algebraic closure shows that the base change of $[v]$ lies in $Y_{d,L,\bar k}$.  Since $Y_{d,L}$ is a closed $k$-subscheme of $\PP_{\mathrm{lin}}(V_d^{\oplus L})$, faithful flatness of $\bar k/k$ implies that $[v]$ itself factors through $Y_{d,L}$.  Hence the projectivization of the nonzero part of $\mathcal B_{d,L}(k)$ is contained in $Y_{d,L}(k)$.

Since $\dim B_d\le b$ and the fibres of $I_{d,L}\to B_d$ have dimension $L(M_d-1)-1$,
\[
        \dim Y_{d,L}
        \le
        \dim I_{d,L}
        \le
        b+L(M_d-1)-1
        =LM_d-1-(L-b).
\]

It remains to bound the total degree of $Y_{d,L}$ polynomially in $d$ and $L$.  Put
\[
        E_{d,L}=\mathcal K_d^{\oplus L},
        \qquad
        R=\operatorname{rank}E_{d,L}=L(M_d-1).
\]
Let
\[
        \rho:I_{d,L}\to B_d
\]
be the projection, and let
\[
        \zeta=c_1(\OO_{I_{d,L}}(1)).
\]
Thus $\zeta=f^*H$, where $H=c_1(\OO_{\PP_{\rm lin}(V_d^{\oplus L})}(1))$.  Put
\[
        r=c_1(\mathcal R_d),
        \qquad
        A=\OO_{I_{d,L}}(1)\otimes \rho^*\mathcal R_d.
\]
The morphism $I_{d,L}\to B_d\times\PP_{\mathrm{lin}}(V_d^{\oplus L})$ is a closed immersion, and $A$ is the restriction of the very ample line bundle $\mathcal R_d\boxtimes\OO(1)$.  Hence $A$ is very ample and
\[
        c_1(A)=\zeta+\rho^*r.
\]

Let $Z$ be an irreducible component of $(Y_{d,L})_{\mathrm{red}}$ of dimension $c$.  Choose an irreducible component $J$ of $(I_{d,L})_{\mathrm{red}}$ whose image closure is $Z$.  The irreducible components of $(I_{d,L})_{\mathrm{red}}$ are precisely the projective bundles over the irreducible components of $(B_d)_{\mathrm{red}}$.  Thus there is a unique irreducible component $C$ of $(B_d)_{\mathrm{red}}$ such that, with $n=\dim C\le b$,
\[
        J=\PP_C\bigl((E_{d,L}|_C)^\vee\bigr),
        \qquad
        \dim J=R-1+n.
\]
The scheme $J$ is integral.  Since the topological image of $J$ is contained in the reduced closed subscheme $Z$, the restriction $f|_J$ factors uniquely through a dominant projective morphism
\[
        g:J\longrightarrow Z.
\]
For every geometric point $x\in C$, the restriction of $f$ to the fibre
\[
        \rho^{-1}(x)\simeq \PP^{R-1}
\]
is the linear closed immersion induced by
\[
        E_{d,L,x}\subset V_d^{\oplus L}\otimes_k\kappa(x).
\]
Because this fibre maps into $Z$, we have $c=\dim Z\ge R-1$.  Hence
\[
        e:=\dim J-\dim Z=R-1+n-c
\]
satisfies
\[
        0\le e\le n\le b.
\]

The generic fibre $J_{\eta_Z}$ has dimension $e$.  Since $A$ is very ample on $J$, its restriction to $J_{\eta_Z}$ is ample, and
\[
        \alpha:=\int_{J_{\eta_Z}}c_1(A)^e
\]
is a positive integer.  The restriction of $c_1(A)^e\cap[J]$ to the generic fibre of $g$ is a zero-cycle of degree $\alpha$ over $k(Z)$.  By the definition of proper pushforward, the coefficient at the generic point of $Z$ in
\[
        g_*\bigl(c_1(A)^e\cap[J]\bigr)
\]
is therefore $\alpha$.  The difference between this class and $\alpha[Z]$ is supported on a proper closed subset of the integral $c$-dimensional scheme $Z$, which has dimension smaller than $c$ and contributes nothing to $A_c(Z)$.  Hence
\[
        g_*\bigl(c_1(A)^e\cap[J]\bigr)=\alpha[Z]
        \qquad\text{in }A_c(Z).
\]
By the projection formula,
\[
\begin{aligned}
        \int_J \zeta^c\,c_1(A)^e
        &=
        \int_J g^*H^c\,c_1(A)^e                                      \\
        &=
        \int_Z H^c\cap g_*\bigl(c_1(A)^e\cap[J]\bigr)                \\
        &=
        \alpha\int_Z H^c
        =\alpha\deg Z
        \ge \deg Z.
\end{aligned}
\]
Consequently,
\[
        \deg Z
        \le
        \int_J \zeta^c(\zeta+\rho^*r)^e
        \le
        \sum_{a=0}^{e}\binom ea
        \left|
        \int_J
        \zeta^{R-1+n-a}(\rho^*r)^a
        \right|,
\]
where we used $c+e=\dim J=R-1+n$.

We retain the convention $s(\mathcal F)=c(\mathcal F)^{-1}$ from Lemma~\ref{lem:length-two-degree}.  Since
\[
        J=\PP_C((E_{d,L}|_C)^\vee)
\]
is a projective bundle of one-dimensional quotients, the projective bundle formula gives, for $0\le a\le e\le n$,
\[
        \rho_*\bigl(\zeta^{R-1+n-a}\cap[J]\bigr)
        =s_{n-a}(E_{d,L}|_C).
\]
Now
\[
        0\to \mathcal K_d
        \to V_d\otimes_k\OO_{B_d}
        \to \mathcal Q_d
        \to0.
\]
Let
\[
        x=c_1(\mathcal Q_d).
\]
The exact sequence gives, in the Chow ring of $B_d$,
\[
        c(\mathcal K_d)=(1+x)^{-1},
        \qquad
        c(E_{d,L})=(1+x)^{-L},
        \qquad
        s(E_{d,L})=(1+x)^L,
\]
where the inverse is expanded in the graded Chow ring and therefore truncates in finite degree.
As $n-a\le b<L$, it follows that
\[
        s_{n-a}(E_{d,L}|_C)=\binom{L}{n-a}x^{n-a}|_C.
\]
Therefore
\[
        \int_J
        \zeta^{R-1+n-a}(\rho^*r)^a
        =
        \binom{L}{n-a}\int_C x^{n-a}r^a.
\]
Also
\[
        r=x+\pi_d^*c_1(\OO_{\mathcal H}(1)).
\]
After expanding $r^a$, the last integral is a finite integral linear combination of intersection numbers on $C$ of monomials of total degree $n$ in
\[
        c_1(\mathcal Q_d),
        \qquad
        \pi_d^*c_1(\OO_{\mathcal H}(1)),
\]
and its coefficients are polynomials in $L$ of degree at most $b$.  By Lemma~\ref{lem:length-two-degree}, the sum of the absolute values of these monomial intersection numbers over all irreducible components $C$ of $(B_d)_{\mathrm{red}}$ is bounded, for $d\gg0$, by a polynomial in $d$.

The number of irreducible components of $(I_{d,L})_{\mathrm{red}}$ is the number of irreducible components of $(B_d)_{\mathrm{red}}$ and is independent of $d$ and $L$.  For each irreducible component $Z$ of $(Y_{d,L})_{\mathrm{red}}$, choose one component $J$ whose image closure is $Z$.  This choice is injective as a map from the set of such $Z$ to the set of components $J$, because a fixed irreducible $J$ has only one image closure.  Since $0\le e\le b$, the binomial coefficients $\binom ea$ are bounded by constants depending only on $b$.  Summing the preceding estimates therefore gives a polynomial $P_1(d,L)$.  After choosing an integer $d_1$, independent of $L$, large enough for Lemma~\ref{lem:length-two-degree}, one has
\[
        \deg^{\mathrm{tot}}Y_{d,L}\le P_1(d,L)
\]
for all $d\ge d_1$ and all $L>b$.

Set
\[
        D_{d,L}=LM_d-1-(L-b).
\]
By Lemma~\ref{lem:finite-field-degree-bound}, applied to $Y_{d,L}\subset\PP_{\mathrm{lin}}(V_d^{\oplus L})$ with its total degree with respect to $\OO(1)$,
\[
        \#Y_{d,L}(k)
        \le C_qP_1(d,L)q^{D_{d,L}}.
\]
The affine cone over $Y_{d,L}$ has exactly $1+(q-1)\#Y_{d,L}(k)$ $k$-points.  Since $\mathcal B_{d,L}(k)$ is contained in this cone, after enlarging the polynomial once more and taking $d_0=d_1$, we obtain
\[
        \#\mathcal B_{d,L}(k)
        \le
        P_0(d,L)q^{D_{d,L}+1}
        =P_0(d,L)q^{LM_d-(L-b)},
\]
as required.
\end{proof}

\begin{lemma}[Linear-dependence count after finite scalar extension]\label{lem:extension-dependence-count}
Let $V$ be an $M$-dimensional $k$-vector space, let $k'/k$ be a finite extension, and let
\[
        \rho:V\otimes_k k'\longrightarrow H
\]
be a $k'$-linear map with image dimension $c$ over $k'$.  Fix $L\le c$.  Among all ordered tuples $(s_1,\ldots,s_L)\in V^L$, the number for which
\[
        \rho(s_1),\ldots,\rho(s_L)
\]
are $k'$-linearly dependent is at most
\[
        Lq^{-(c-L+1)}q^{ML}.
\]
\end{lemma}

\begin{proof}
Count according to the first position at which dependence appears.  Suppose that the first $t$ images are $k'$-linearly independent, with $0\le t\le L-1$, and let $S_t$ be their $k'$-span.  The next image is dependent precisely when
\[
        \rho(s_{t+1})\in S_t.
\]
The set of $v\in V$ satisfying this condition is a $k$-linear subspace.  After quotienting $\im\rho$ by $S_t$, the $k'$-span of the image of $V$ has dimension $c-t$, so the corresponding $k$-linear map has image of $k$-dimension at least $c-t$.  Therefore the subspace has codimension at least $c-t$ in $V$.  For fixed first $t$ entries, there are at most $q^{M-(c-t)}$ choices for $s_{t+1}$ and arbitrary choices for the remaining entries.  The number of tuples whose first dependence occurs at position $t+1$ is at most
\[
        q^{Mt}q^{M-(c-t)}q^{M(L-t-1)}=q^{ML-(c-t)}.
\]
Summing over $t=0,\ldots,L-1$ gives
\[
        \sum_{t=0}^{L-1}q^{ML-(c-t)}
        \le Lq^{ML-(c-L+1)}.
\]
\end{proof}

\begin{lemma}[Finite disjoint unions of closed immersions]\label{lem:finite-disjoint-closed-immersion}
Let $Z=\coprod_{\alpha=1}^n Z_\alpha$ be a finite disjoint union of finite $k$-schemes.  Suppose that $f_\alpha:Z_\alpha\to\PP^r_k$ is a closed immersion for every $\alpha$, and that the supports of the scheme-theoretic images $f_\alpha(Z_\alpha)$ are pairwise disjoint.  Then the induced morphism
\[
        f:Z\longrightarrow\PP^r_k
\]
is a closed immersion.
\end{lemma}

\begin{proof}
Each image $f_\alpha(Z_\alpha)$ is a finite closed subscheme of $\PP^r_k$.  Since their supports are pairwise disjoint, their scheme-theoretic union has structure sheaf equal to the product of the individual structure sheaves.  Thus this union is canonically isomorphic to $\coprod_\alpha Z_\alpha$.  The morphism $f$ identifies $Z$ with this closed subscheme of $\PP^r_k$, so $f$ is a closed immersion.
\end{proof}

\begin{lemma}[Closed immersions separate length-two subschemes]\label{lem:closed-immersion-separates-length-two}
Let $Y$ be a finite type $k$-scheme, let $\mathcal L$ be a line bundle, and let $W\subset H^0(Y,\mathcal L)$ be a basepoint-free finite-dimensional linear system.  If the morphism $\phi_W:Y\to\PP(W)$ is a locally closed embedding, then for every algebraically closed extension $\Omega/k$ and every closed subscheme $\xi\subset Y_\Omega$ of length at most two, the restriction map
\[
        W_\Omega\to H^0(\xi,\mathcal L_\Omega|_\xi)
\]
is surjective.
\end{lemma}

\begin{proof}
After base change to $\Omega$, the morphism
\[
        \phi_{W,\Omega}:Y_\Omega\to\PP(W_\Omega)
\]
is still a locally closed immersion.  Let $\xi\subset Y_\Omega$ be a closed subscheme of length at most two.  The composite
\[
        \xi\longrightarrow Y_\Omega\longrightarrow\PP(W_\Omega)
\]
is a locally closed immersion.  Since $\xi$ is finite over $\Omega$, it is proper over $\Omega$; because projective space is separated over $\Omega$, the composite is proper.  A proper locally closed immersion is a closed immersion.  Thus $\xi$ is identified with a closed subscheme of $\PP(W_\Omega)$, and $\mathcal L_\Omega|_\xi$ is the pullback of $\OO_{\PP(W_\Omega)}(1)$.

It remains to prove that linear forms on projective space restrict onto every closed subscheme of length at most two.  Let $\eta\subset\PP^n_\Omega$ be such a subscheme.  The cases of length zero and one are immediate.  Suppose that $\eta$ has length two.  Choose a linear form $\ell_0$ that is nonzero at every point of the support of $\eta$.  The closed immersion $\eta\hookrightarrow D_+(\ell_0)\simeq\mathbb A^n_\Omega$ implies that the affine linear functions $\ell/\ell_0$ generate $H^0(\eta,\OO_\eta)$ as an $\Omega$-algebra.  Since this algebra has $\Omega$-dimension two, the unit and one such affine linear function span it.  Trivializing $\OO(1)|_\eta$ by $\ell_0$ therefore gives a surjection
\[
        H^0(\PP^n_\Omega,\OO(1))\twoheadrightarrow H^0(\eta,\OO(1)|_\eta).
\]
Equivalently, the scheme-theoretic projective span of $\eta$ is a line.  This proves the required surjectivity for $\xi$.
\end{proof}

\begin{lemma}[Compactified two-block embedding]\label{lem:two-block-separation}
Let $Y$ be a projective $k$-scheme, let $S\subset |Y|$ be a finite set contained in the smooth locus of $Y$, and put
\[
        Z=\coprod_{P\in S}\Spec(\OO_{Y,P}/\mm_P^2).
\]
Let $u_1,\ldots,u_A,w_1,\ldots,w_B$ be global sections of a line bundle $\mathcal L$ on $Y$.  Assume:
\begin{alphlist}
\item the $w$-block has no common zero on $Y\setminus S$ and defines a locally closed embedding $Y\setminus S\to\PP^{B-1}_k$;
\item each $w_j$ vanishes on $Z$;
\item the $u$-block defines a closed immersion $Z\to\PP^{A-1}_k$;
\item the full collection $u_1,\ldots,u_A,w_1,\ldots,w_B$ has no common zero on $Y$.
\end{alphlist}
Then the full collection defines a closed immersion
\[
        Y\hookrightarrow\PP^{A+B-1}_k .
\]
\end{lemma}

\begin{proof}
Let $\Omega/k$ be algebraically closed and let $\xi\subset Y_\Omega$ have length two.  The no-common-zero hypotheses in (a) and (d) are the surjectivity of the corresponding evaluation maps to $\mathcal L$ on length-one subschemes; these surjections are preserved after base change to $\Omega$.

If $\xi\subset (Y\setminus S)_\Omega$, then the $w$-block alone gives a surjection onto $H^0(\xi,\mathcal L|_\xi)$ by Lemma~\ref{lem:closed-immersion-separates-length-two}.  If $\xi$ is supported over $S$, then $\xi\subset Z_\Omega$.  Indeed, for each $P\in S$ the extension $\kk(P)/k$ is finite separable, so $(\Spec(\OO_{Y,P}/\mm_P^2))_\Omega$ is the disjoint union of the first infinitesimal neighbourhoods of the geometric points of $P_\Omega$.  Since $Y$ is smooth at $P$, any local length-two subscheme supported at such a geometric point has defining ideal containing the square of the maximal ideal, and a reduced length-two subscheme supported over $S$ is contained in the corresponding disjoint union of reduced points.  Thus every length-two subscheme supported over $S$ is contained in $Z_\Omega$, and the $u$-block gives the required surjection by the same lemma.

The only remaining case is
\[
        \xi=\{x\}\amalg\{y\},
        \qquad
        x\in S_\Omega,
        \quad
        y\in (Y\setminus S)_\Omega.
\]
All $w_j$ vanish at $x$, while the base-changed $w$-block has no common zero at $y$; hence some $w_j$ is nonzero at $y$.  Thus the image of the full restriction map contains an element of $\mathcal L_x\oplus\mathcal L_y$ of the form $(0,\lambda)$ with $\lambda\ne0$.  By the base-changed no-common-zero hypothesis for the full collection and the vanishing of the $w_j$ at $x$, some $u_\ell$ is nonzero at $x$.  Subtracting a suitable scalar multiple of the previous $w$-section from $u_\ell$, we obtain a section whose value is nonzero at $x$ and zero at $y$.  These two elements span $\mathcal L_x\oplus\mathcal L_y$, so the restriction map is surjective.

Thus the span $W$ of the full collection separates every geometric length-two subscheme of $Y$.  Hypothesis (d) says that $W$ is basepoint-free.  Lemma~\ref{lem:linear-system-criterion} gives a closed immersion
\[
        Y\hookrightarrow\PP(W).
\]
The chosen ordered generating set induces a surjection $k^{A+B}\twoheadrightarrow W$ and hence a linear closed immersion $\PP(W)\hookrightarrow\PP(k^{A+B})=\PP^{A+B-1}_k$.  The displayed morphism is the composition of these two closed immersions, and is therefore a closed immersion.
\end{proof}

\begin{lemma}[A nonzero section with zero first jet is singular]\label{lem:first-order-singular}
Let $(R,\mathfrak m)$ be a regular local domain of dimension $m\ge1$.  If $0\ne f\in\mathfrak m^2$, then $R/(f)$ is not regular.  Consequently, if $X/k$ is smooth, $\mathcal L$ is a line bundle on $X$, $P\in X$ is a closed point, and a section $s\in H^0(X,\mathcal L)$ has nonzero germ represented by $f\in\mathfrak m_P^2$ after a local trivialization, then the zero scheme $V(s)$ is not smooth at $P$.
\end{lemma}

\begin{proof}
A regular local ring is Cohen--Macaulay.  Since $R$ is a domain and $f\ne0$, the element $f$ is a non-zero-divisor; since $f\in\mathfrak m^2\subset\mathfrak m$, it is not a unit.  Quotienting a Cohen--Macaulay local ring by one non-zero-divisor lowers the dimension by one, so
\[
        \dim R/(f)=m-1.
\]
Let $\mathfrak n=\mathfrak m/(f)$ be the maximal ideal of $R/(f)$.  Because $f\in\mathfrak m^2$,
\[
        \mathfrak n/\mathfrak n^2
        \simeq
        \mathfrak m/(\mathfrak m^2+(f))
        =
        \mathfrak m/\mathfrak m^2.
\]
Since $R$ is regular of dimension $m$, the last vector space has dimension $m$ over the residue field.  Hence
\[
        \edim R/(f)=m>m-1=\dim R/(f),
\]
so $R/(f)$ is not regular.  For a smooth $X/k$, the local ring $\OO_{X,P}$ is regular; the local ring of $V(s)$ at $P$ is $\OO_{X,P}/(f)$.  Since smooth schemes over fields are regular, nonregularity implies nonsmoothness at $P$.
\end{proof}

The following auxiliary fact is used only to fix one embedding system.  The proof is written in an ordered-tuple form, so the finite-field avoidance is an explicit point count rather than an appeal to the false principle that a nonempty open set over a finite field must have a rational point.

\begin{lemma}[Componentwise nondegenerate compactified initial embedding]\label{lem:initial}
Assume $k=\F_q$ and that $X$ is a nonempty smooth quasiprojective $k$-scheme of pure dimension $m\ge1$.  Then there exist a projective closure $j:X\hookrightarrow\barX$ such that $X$ is open in $\barX$ and $\barX$ is reduced and of pure dimension $m$, a line bundle $\overline{\EE}$ on $\barX$, and sections
\[
        \overline e_0,\ldots,\overline e_r\in H^0(\barX,\overline{\EE})
\]
such that
\[
        \overline e:\barX\longrightarrow\PP^r_k,
        \qquad
        x\longmapsto[\overline e_0(x):\cdots:\overline e_r(x)]
\]
is a closed immersion.  If
\[
        \EE=j^*\overline{\EE},
        \qquad
        e_i=\overline e_i|_X,
\]
then
\[
        e:X\longrightarrow\PP^r_k,
        \qquad
        x\longmapsto[e_0(x):\cdots:e_r(x)]
\]
is a locally closed embedding and, for every connected component $C$ of $X$, if $k_C$ denotes the algebraic closure of $k$ in $k(C)$, then the restrictions
\[
        e_0|_C,\ldots,e_r|_C
\]
are $k_C$-linearly independent.  Consequently, for every field extension $K/k$ and every connected component $D$ of $C_K$, the image of $D$ is not contained in any hyperplane of $\PP^r_K$.
\end{lemma}

\begin{proof}
Choose a locally closed immersion of $X$ into a projective space and let $\barX$ be the reduced scheme-theoretic closure of its image.  Then $X$ is an open subscheme of $\barX$, every irreducible component of $\barX$ meets $X$, and $\barX$ is projective, reduced, and of pure dimension $m$.  Choose a very ample line bundle $\Ample$ on $\barX$.  Put
\[
        V_d=H^0(\barX,\Ample^{\otimes d}),
        \qquad
        M_d=\dim_kV_d.
\]
For every $d\ge1$, the complete system $V_d$ gives a projective embedding of $\barX$.

Let $C_1,\ldots,C_s$ be the connected components of $X$, and put $k_\mu=k_{C_\mu}$.  By Lemma~\ref{lem:component-constants}, every $C_\mu$ is integral and geometrically integral over $k_\mu$.  Define
\[
        \rho_{\mu,d}:V_d\otimes_k k_\mu
        \longrightarrow
        H^0(C_\mu,j^*\Ample^{\otimes d}|_{C_\mu})
\]
and let
\[
        c_{\mu,d}=\dim_{k_\mu}\im(\rho_{\mu,d}).
\]
We first show that $c_{\mu,d}\to\infty$ for every $\mu$.  Since $\Ample$ is very ample and $C_\mu$ has positive dimension, there are sections $s,t\in H^0(\barX,\Ample)$ such that $s|_{C_\mu}$ is not identically zero and $t/s$ is a nonconstant rational function on $C_\mu$.  Indeed, the morphism defined by $H^0(\barX,\Ample)$ is a closed immersion of $\barX$ into a projective space; the closure of $C_\mu$ has positive dimension, so its image is not a single point.  Two homogeneous coordinates with nonconstant ratio on this image, with the denominator not identically zero on $C_\mu$, give such $s$ and $t$.  If
\[
        \sum_{a=0}^d \lambda_a s^{d-a}t^a=0
        \qquad(\lambda_a\in k_\mu)
\]
on $C_\mu$, then on the nonempty open where $s\ne0$ we have
\[
        \sum_{a=0}^d \lambda_a(t/s)^a=0.
\]
The element $t/s\in k(C_\mu)$ is nonconstant and therefore is not algebraic over the finite field $k_\mu$, because $k_\mu$ is algebraically closed in $k(C_\mu)$.  Hence all $\lambda_a$ are zero.  Thus the $d+1$ restrictions
\[
        s^d,s^{d-1}t,\ldots,t^d
\]
are $k_\mu$-linearly independent, and $c_{\mu,d}\ge d+1$.

Next consider the compactified length-two separation condition.  The scheme $\mathcal H=\Hilb^2(\barX)$ is nonempty: after base change to an algebraic closure, a closed point on any positive-dimensional irreducible component of $\barX$ admits a length-two infinitesimal neighbourhood.  Construct $B_d=\PP_{\mathcal H}(\mathcal E_d)$ and its quotient line bundle $\mathcal Q_d$ as in Lemma~\ref{lem:length-two-degree}.  Put
\[
        b=b_{\barX}=\dim\mathcal H+1.
\]
By Lemma~\ref{lem:length-two-bad-locus}, for $L>b$ the proportion of ordered $L$-tuples in $V_d^L$ failing to separate some geometric length-two subscheme of $\barX$ is at most
\[
        P_0(d,L)q^{-(L-b)},
\]
where $P_0(d,L)$ is polynomial in $d$ and $L$, and the lower bound on $d$ in that lemma is independent of $L$.

For componentwise linear independence, Lemma~\ref{lem:extension-dependence-count} gives, for fixed $\mu$ and $L\le c_{\mu,d}$, that the proportion of ordered $L$-tuples whose restrictions to $C_\mu$ are $k_\mu$-linearly dependent is at most
\[
        Lq^{-(c_{\mu,d}-L+1)}.
\]
The same lemma, applied to the identity map $V_d\to V_d$, shows that the proportion of ordered $L$-tuples that are $k$-linearly dependent in $V_d$ is at most
\[
        Lq^{-(M_d-L+1)}.
\]

Set
\[
        c_d=\min_\mu c_{\mu,d},
        \qquad
        L_d=\left\lfloor \frac{c_d}{2}\right\rfloor .
\]
For $d\gg0$ we have $L_d>b$ and $L_d\to\infty$.  Since $c_{\mu,d}\ge c_d$ and $L_d=\lfloor c_d/2\rfloor$, we have
\[
        c_{\mu,d}-L_d\ge c_d-\left\lfloor\frac{c_d}{2}\right\rfloor\longrightarrow\infty.
\]
Moreover $c_d\le M_d$ and $c_d\ge d+1$, so
\[
        L_d\le\frac{c_d}{2}\le\frac{M_d}{2},
        \qquad
        M_d-L_d\ge\frac{M_d}{2}\longrightarrow\infty.
\]
The Hilbert function $M_d$ is eventually a polynomial in $d$, while $L_d\le M_d$ and $L_d\ge(d+1)/2-1$.  Hence $P_0(d,L_d)$ and the other polynomial factors are dominated by the exponential powers of $q$, and for $d\gg0$ one has
\[
        L_dq^{-(M_d-L_d+1)}
        +P_0(d,L_d)q^{-(L_d-b)}
        +\sum_{\mu=1}^s L_dq^{-(c_{\mu,d}-L_d+1)}
        <1.
\]
Thus there exists an ordered tuple
\[
        \overline e_0,\ldots,\overline e_{L_d-1}\in V_d
\]
which is $k$-linearly independent, separates every geometric length-two subscheme of $\barX$, and whose restrictions to every $C_\mu$ are $k_\mu$-linearly independent.  Let
\[
        W=\Span_k(\overline e_0,\ldots,\overline e_{L_d-1}),
        \qquad
        r=L_d-1,
        \qquad
        \overline{\EE}=\Ample^{\otimes d}.
\]
The length-two separation condition implies that $W$ is basepoint-free on $\barX$.  Indeed, if the common zero locus of $W$ on $\barX$ were nonempty, then, because $\barX$ is projective over a field, it would contain a closed point $x$.  After extension of $\kk(x)$ to an algebraic closure $\Omega$, the point $x$ becomes a geometric point of $\barX_\Omega$.  Since every irreducible component of $\barX$ has positive dimension, the local ring of $\barX_\Omega$ at this geometric point is not a field, and its maximal ideal has nonzero cotangent space.  Hence there exists a closed length-two subscheme
\[
        \xi\subset\barX_\Omega
\]
supported at this point.  All sections in $W_\Omega$ vanish at the support, so their restrictions to $\xi$ lie in the one-dimensional nilpotent ideal of $H^0(\xi,\overline{\EE}|_\xi)$ and cannot span this two-dimensional vector space.  This contradicts the chosen length-two separation.

Therefore Lemma~\ref{lem:linear-system-criterion} applies to the projective scheme $\barX$, and the morphism
\[
        \overline e:\barX\to\PP(W)\simeq\PP^r_k
\]
is a closed immersion.  Its restriction $e:X\to\PP^r_k$ is a locally closed embedding.  The componentwise $k_\mu$-linear independence is part of the construction, and the assertion after arbitrary scalar extension follows from Lemma~\ref{lem:component-base-change}.
\end{proof}

\section{Local codes and shielded sections}\label{sec:local-shield}

The next lemma is the local coding device.  It explains the inequality $A\ge m+3$: after imposing one $k$-linear relation among $A$ homogeneous coordinates, enough affine coordinates remain to encode both the residue field and the $m$ first-order tangent parameters.

\begin{lemma}[A first-order code with one prescribed relation]\label{lem:localcode}
Fix integers $m\ge0$ and $A\ge m+3$, and let $a=(a_1,\ldots,a_A)\in k^A$.  For every finite extension $K/k$ and every square-zero $K$-algebra
\[
        J=K\oplus M,
        \qquad
        M^2=0,
        \qquad
        \dim_K M=m,
\]
there exist elements
\[
        \eta_1,\ldots,\eta_A\in J
\]
with the following properties:
\begin{alphlist}
\item if $a\ne0$, then
\[
        \sum_{\ell=1}^{A}a_\ell\eta_\ell=0;
\]
\item the morphism
\[
        \Spec J\longrightarrow\PP^{A-1}_k,
        \qquad
        z\longmapsto [\eta_1(z):\cdots:\eta_A(z)]
\]
is a closed immersion.
\end{alphlist}
Moreover, for every finite set $F\subset |\PP^{A-1}_k|$ there is an integer
\[
        d_0=d_0(F,a,A,m)
\]
such that, for every finite extension $K/k$ with $[K:k]\ge d_0$ and every square-zero $K$-algebra $J=K\oplus M$ as above, the elements $\eta_\ell$ may be chosen to satisfy \textup{(a)} and \textup{(b)} and so that the closed point $\Spec K\to\PP^{A-1}_k$ defined by the residue value vector
\[
        (\overline\eta_1,\ldots,\overline\eta_A)\in K^A
\]
does not lie in $F$.
\end{lemma}

\begin{proof}
First suppose $a\ne0$.  Regard $a$ as a row vector and choose $G\in\operatorname{GL}_A(k)$ such that
\[
        aG=(1,0,\ldots,0).
\]
It is enough to construct coordinates $\xi_1,\ldots,\xi_A$ satisfying $\xi_1=0$ and then set
\[
        (\eta_1,\ldots,\eta_A)^t=G(\xi_1,\ldots,\xi_A)^t;
\]
the projective automorphism induced by $G$ preserves closed immersions, and
\[
        \sum a_\ell\eta_\ell=aG\xi=\xi_1=0.
\]
Choose an element $\theta\in K$ that generates $K$ over $k$, and choose a $K$-basis $\tau_1,\ldots,\tau_m$ of $M$.  Define
\[
        \xi_1=0,
        \qquad
        \xi_2=1,
        \qquad
        \xi_3=\theta,
        \qquad
        \xi_{3+i}=\tau_i\quad(1\le i\le m),
\]
and set the remaining $\xi_\ell$ equal to $0$.  On the affine chart $X_2\ne0$ the induced $k$-algebra map has image containing $\theta,\tau_1,\ldots,\tau_m$.  Hence it contains $K=k(\theta)$ and the ideal $M=K\tau_1+\cdots+K\tau_m$, so it is surjective onto $J=K\oplus M$.  Thus the morphism $\Spec J\to\{X_2\ne0\}$ is a closed immersion.  Its composite with the open immersion $\{X_2\ne0\}\hookrightarrow\PP^{A-1}_k$ is an immersion.  Since $\Spec J$ is finite, hence proper, over $k$ and $\PP^{A-1}_k$ is separated over $k$, this composite is proper; a proper immersion is a closed immersion.  This proves (a) and (b) when $a\ne0$.

If $a=0$, use instead
\[
        \xi_1=1,
        \xi_2=\theta,
        \xi_{2+i}=\tau_i\quad(1\le i\le m),
\]
with the remaining coordinates equal to $0$.  The same affine-chart argument, followed by the same proper-immersion argument, gives a closed immersion, and there is no relation to impose.

For the uniform avoidance assertion, first pull the forbidden finite set back by the projective automorphism induced by $G$ in the case $a\ne0$; in the case $a=0$ leave it unchanged.  In these transformed coordinates the construction places the residue point on a fixed projective line $L$ over $k$ whose affine coordinate is $\theta$.  The transformed finite forbidden set meets $L$ in finitely many closed points.  Let $D_F$ be the maximum of their degrees, with $D_F=0$ if the intersection is empty.  Because $\theta$ generates $K$ over $k$, the resulting point of $L$ has residue field $k(\theta)=K$ and therefore degree $[K:k]$.  Hence, if $[K:k]>D_F$, this point cannot belong to the transformed forbidden set.  Taking
\[
        d_0=D_F+1
\]
and then applying the inverse projective linear change gives the required avoidance in the original coordinates.  The integer $d_0$ depends only on $F$ and the change of coordinates determined by $a$, and is independent of $K$ and $M$.
\end{proof}

We next isolate the shielded sections.  The binary vector prescribed at each test point is required to be nonzero; this is necessary and is exactly what will be provided by Lemma~\ref{lem:binary}.

\begin{lemma}[Compactified shielded sections with prescribed binary values]\label{lem:shield}
Let $X$ be a smooth quasiprojective $k$-scheme of pure dimension $m\ge1$, and let
\[
        j:X\hookrightarrow\barX
\]
be a projective closure such that $X$ is open in $\barX$ and $\barX$ is of pure dimension $m$.  Let $\Ample$ be an ample line bundle on $\barX$.  Let $S\subset |X|$ be finite, and put
\[
        Z=\coprod_{P\in S}P^{(2)}.
\]
Let
\[
        \mathcal R=\set{R_1,\ldots,R_t}\subset |X|\setminus S
\]
be finite.  For each $R_\nu$ fix a nonzero vector
\[
        \epsilon_\nu=(\epsilon_{\nu0},\ldots,\epsilon_{\nu m})
        \in\set{0,1}^{m+1}\setminus\{0\}
        \subset \kk(R_\nu)^{m+1}.
\]
Then there exist an integer $D>0$, trivializations of $\Ample^{\otimes D}$ at the points of $\mathcal R$, and sections
\[
        \overline h_0,\ldots,\overline h_m\in H^0(\barX,\Ample^{\otimes D})
\]
such that, writing $h_i=\overline h_i|_X$ and using the induced trivializations on $X$,
\begin{align}
        h_i|_Z&=0,                                                    \label{eq:shield-fat}\\
        h_i(R_\nu)&=\epsilon_{\nu i}       &&(1\le\nu\le t),          \label{eq:shield-values}\\
        \bigcap_{i=0}^{m}V_{\barX}(\overline h_i)&=S                  \label{eq:shield-zero}
\end{align}
set-theoretically.  Moreover, replacing all $\overline h_i$ by $\overline h_i^n$ preserves \eqref{eq:shield-fat}--\eqref{eq:shield-zero} and replaces $\Ample^{\otimes D}$ by the arbitrarily positive line bundle $\Ample^{\otimes Dn}$.
\end{lemma}

\begin{proof}
Choose arbitrary trivializations of $\Ample$ at the points of $\mathcal R$; these induce trivializations of all tensor powers at those points.  The finite subschemes $Z$ and the finite reduced subscheme supported on $\mathcal R$ are finite over $k$; composed with $j$ they are proper locally closed subschemes of the projective $k$-scheme $\barX$, hence closed subschemes of $\barX$.  Every auxiliary point chosen below is a closed point of $\barX$, and every finite reduced set of such points is a finite closed subscheme of $\barX$.  We construct sections $g_i\in H^0(\barX,\Ample^{\otimes d_i})$ with the desired vanishing on $Z$, the prescribed binary values on $\mathcal R$, and whose common zero set on $\barX$ is exactly $S$.

Put
\[
        Y_{-1}=\barX\setminus S.
\]
Assume $g_0,\ldots,g_{i-1}$ have already been chosen and set
\[
        Y_{i-1}=(\barX\setminus S)\cap V_{\barX}(g_0,\ldots,g_{i-1}).
\]
Let $\delta_{i-1}=\dim Y_{i-1}$, with $\delta_{i-1}=-\infty$ if $Y_{i-1}=\varnothing$.  If $\delta_{i-1}>0$, list the irreducible components of $Y_{i-1}$ of dimension $\delta_{i-1}$ as $\Gamma_1,\ldots,\Gamma_a$ and choose closed points
\[
        y_b\in\Gamma_b\setminus \mathcal R
        \qquad(1\le b\le a),
\]
which is possible by Lemma~\ref{lem:closedpoints}.  If $\delta_{i-1}=0$, let $y_1,\ldots,y_a$ be the finitely many closed points of $Y_{i-1}$ that are not in $\mathcal R$.  If $Y_{i-1}=\varnothing$, take no auxiliary points.  The points $y_b$ are disjoint from $S\cup\mathcal R$.  Choose arbitrary trivializations of the relevant tensor powers of $\Ample$ at these auxiliary points as well.

By Lemma~\ref{lem:interpolation-projective}, for $d_i\gg0$ there exists
\[
        g_i\in H^0(\barX,\Ample^{\otimes d_i})
\]
whose restriction to the finite subscheme $Z$ is zero, whose values at the points $R_\nu$ are $\epsilon_{\nu i}$, and whose values at all auxiliary points $y_b$ are $1$.  If $\delta_{i-1}>0$, no irreducible component of $Y_{i-1}$ of maximal dimension is contained in $V_{\barX}(g_i)$; hence
\[
        \dim Y_i<\dim Y_{i-1}.
\]
If $\delta_{i-1}=0$, every point of $Y_{i-1}\setminus\mathcal R$ is removed.

Starting from $\dim Y_{-1}=m$, after choosing $g_0,\ldots,g_{m-1}$ no positive-dimensional component remains.  The last section $g_m$ removes every remaining zero-dimensional point outside $\mathcal R$.  A point $R_\nu\in\mathcal R$ cannot remain in the final common zero set because the vector $\epsilon_\nu$ is nonzero, so $g_i(R_\nu)=1$ for at least one $i$.  Therefore
\[
        \bigcap_{i=0}^{m}V_{\barX}(g_i)=S.
\]
This equality is set-theoretic on the underlying topological space of $\barX$: since $\barX$ is of finite type over a field, it is Jacobson, and any nonempty closed subset of $\barX\setminus S$ would contain a closed point of $\barX\setminus S$, all of which have been excluded by the construction above.

Let $D$ be a common multiple of $d_0,\ldots,d_m$ and define
\[
        \overline h_i=g_i^{D/d_i}\in H^0(\barX,\Ample^{\otimes D}).
\]
Since $0^e=0$ and $1^e=1$, the vanishing on $Z$, the binary values on $\mathcal R$, and the set-theoretic zero locus are unchanged.  The final assertion follows for the same reason after replacing $\overline h_i$ by $\overline h_i^n$.
\end{proof}

The next lemma supplies the binary data.  Since $X$ is smooth, each connected component is regular; because irreducible components of a regular noetherian scheme are open and closed, each connected component is integral.

\begin{lemma}[Binary values forcing componentwise independence]\label{lem:binary}
Let $X$ be a smooth quasiprojective $k$-scheme of pure dimension $m\ge1$, with connected components $C_1,\ldots,C_s$, and for each $C_\mu$ let $k_\mu$ be the algebraic closure of $k$ in $k(C_\mu)$.  Let
\[
        e_0,\ldots,e_r\in H^0(X,\EE)
\]
be sections of a line bundle $\EE$ such that the morphism
\[
        e:X\longrightarrow\PP^r_k,
        \qquad
        x\longmapsto[e_0(x):\cdots:e_r(x)]
\]
is a locally closed embedding and, for every $\mu$,
\[
        e_0|_{C_\mu},\ldots,e_r|_{C_\mu}
\]
are $k_\mu$-linearly independent.  Let $S\subset |X|$ be finite.  Then there exist a finite set
\[
        \mathcal R\subset |X|\setminus S
\]
and nonzero binary vectors
\[
        \epsilon_Q=(\epsilon_{Q0},\ldots,\epsilon_{Qm})
        \in\set{0,1}^{m+1}\setminus\{0\}
        \subset \kk(Q)^{m+1}
        \qquad(Q\in\mathcal R)
\]
with the following stronger detection property.  For every connected component $C=C_\mu$ and every nonzero matrix
\[
        \gamma=(\gamma_{ij})\in k_\mu^{(m+1)(r+1)},
\]
there is a point
\[
        Q_{C,\gamma}\in\mathcal R\cap C
\]
such that, after prescribing $h_i(Q)=\epsilon_{Qi}$ at all $Q\in\mathcal R$ with respect to trivializations of the fibres of a line bundle $\MM$, one has
\[
        \left(\sum_{i=0}^{m}\sum_{j=0}^{r}\gamma_{ij}h_i e_j\right)(Q_{C,\gamma})\ne0.
\]
In particular, for any such $h_0,\ldots,h_m$, the products
\[
        h_i e_j,
        \qquad
        0\le i\le m,
        \quad
        0\le j\le r,
\]
are $k_\mu$-linearly independent on every connected component $C_\mu$.
\end{lemma}

\begin{proof}
By Lemma~\ref{lem:component-constants}, every connected component $C_\mu$ is integral and positive-dimensional, and the field $k_\mu$ is finite.  Fix a component $C=C_\mu$ and write $k_C=k_\mu$.  The set $k_C^{(m+1)(r+1)}$ is finite.  For every nonzero matrix
\[
        \gamma=(\gamma_{ij})\in k_C^{(m+1)(r+1)}\setminus\set{0},
\]
define actual sections on the $k_C$-scheme $C$ by
\[
        f_i^\gamma=\sum_{j=0}^{r}\gamma_{ij}e_j|_C
        \in H^0(C,\EE|_C)
        \qquad(0\le i\le m),
\]
where $H^0(C,\EE|_C)$ is viewed as a $k_C$-vector space.  Since the sections $e_j|_C$ are $k_C$-linearly independent, not all $f_i^\gamma$ are zero.  Let
\[
        V_\gamma=\bigcap_{\{i\mid f_i^\gamma\ne0\}}V_C(f_i^\gamma).
\]
This is a proper closed subset of the integral scheme $C$.  Hence
\[
        U_\gamma=C\setminus V_\gamma
\]
is a nonempty open subscheme of the positive-dimensional integral scheme $C$, and is again finite type with a positive-dimensional irreducible component.  Apply Lemma~\ref{lem:closedpoints} to $U_\gamma$ and to the finite set consisting of $S\cap |U_\gamma|$ together with all previously chosen points in $U_\gamma$.  This gives a closed point
\[
        Q_{C,\gamma}\in U_\gamma\setminus S
\]
outside all previously chosen points.  By construction,
\[
        (f_0^\gamma(Q_{C,\gamma}),\ldots,f_m^\gamma(Q_{C,\gamma}))
        \ne0.
\]
Choose a standard basis vector
\[
        \epsilon^{C,\gamma}\in\set{0,1}^{m+1}\setminus\{0\}
\]
corresponding to a nonzero coordinate of this vector, and prescribe
\[
        \epsilon_{Q_{C,\gamma},i}=\epsilon_i^{C,\gamma}.
\]
Let $\mathcal R$ be the union of all points $Q_{C,\gamma}$ over all connected components and all nonzero matrices $\gamma$.  This set is finite, because there are finitely many connected components and finitely many nonzero matrices over each finite field $k_C$.

Now let $h_0,\ldots,h_m$ be sections of a line bundle $\MM$ whose fibres at the points of $\mathcal R$ have been trivialized and whose prescribed values are the above binary vectors.  At $Q_{C,\gamma}$ the chosen trivialization of $\MM$ identifies the value in the fibre of $\MM\otimes\EE$ with
\[
        \left(\sum_{i,j}\gamma_{ij}h_i e_j\right)(Q_{C,\gamma})
        =\sum_i\epsilon_i^{C,\gamma}f_i^\gamma(Q_{C,\gamma})
        \in \EE|_{Q_{C,\gamma}}.
\]
The right-hand side is the previously chosen nonzero coordinate of the vector $(f_0^\gamma(Q_{C,\gamma}),\ldots,f_m^\gamma(Q_{C,\gamma}))$.  Thus no nontrivial $k_C$-linear combination of the products $h_i e_j$ vanishes identically on $C$.
\end{proof}
\section{The two coordinate blocks}\label{sec:blocks}

Fix, once and for all, a projective closure
\[
        j:X\hookrightarrow\barX,
\]
a line bundle $\overline{\EE}$ on $\barX$, and sections $\overline e_0,\ldots,\overline e_r$ as in Lemma~\ref{lem:initial}.  Write
\[
        \EE=j^*\overline{\EE},
        \qquad
        e_i=\overline e_i|_X .
\]
Put
\[
        B=(m+1)(r+1).
\]
For the rest of the proof of Theorem~\ref{thm:main}, fix
\[
        N\ge B+m+2,
        \qquad
        A=N+1-B.
\]
Then
\[
        A\ge m+3,
        \qquad
        N+1=A+(m+1)(r+1).
\]
Set
\[
        \Lambda=(\PP^N_k)^\vee(k).
\]
For each $\lambda\in\Lambda$, choose once and for all a representative coefficient vector
\[
        \lambda=[a:b],
\]
where
\[
        a=(a_1,\ldots,a_A)\in k^A,
        \qquad
        b=(b_{ij})_{0\le i\le m,\,0\le j\le r}
        \in k^{(m+1)(r+1)},
\]
not both zero.  Multiplying this representative by a scalar only multiplies the associated section by that scalar and does not change any hyperplane section.  With this fixed representative, let $H_\lambda$ be the hyperplane defined by
\[
        L_\lambda=
        \sum_{\ell=1}^{A}a_\ell X_\ell+
        \sum_{i=0}^{m}\sum_{j=0}^{r}b_{ij}Y_{ij}.
\]

\subsection{Assigned points and local anti-codes}

Enumerate $\Lambda=\{\lambda_1,\ldots,\lambda_M\}$.  We choose, recursively, pairwise distinct closed points
\[
        P_\lambda\in X
\]
and first-order code vectors
\[
        \xi_{\lambda1},\ldots,\xi_{\lambda A}
        \in J_\lambda:=H^0(P_\lambda^{(2)},\OO_{P_\lambda^{(2)}}).
\]
Suppose this has been done for $\lambda_1,\ldots,\lambda_{\nu-1}$, and let $E_{\nu-1}$ be the finite set of the corresponding closed point images in $\PP^{A-1}_k$.  Choose $P_{\lambda_\nu}\in X$ of sufficiently large degree, outside the previously chosen points, using Lemma~\ref{lem:closedpoints}.  Since $\kk(P_{\lambda_\nu})/k$ is finite separable and $X$ is smooth at $P_{\lambda_\nu}$, the quotient
\[
        J_{\lambda_\nu}\longrightarrow \kk(P_{\lambda_\nu})
\]
admits a $k$-algebra section by formal etaleness of finite separable extensions.  After choosing such a section, we identify
\[
        J_{\lambda_\nu}=K\oplus M,
        \qquad
        K=\kk(P_{\lambda_\nu}),
        \qquad
        M=\mm_{P_{\lambda_\nu}}/\mm_{P_{\lambda_\nu}}^2,
\]
with $M^2=0$ and $\dim_KM=m$.  Applying Lemma~\ref{lem:localcode} to the vector $a$ attached to $\lambda_\nu$, and choosing the degree of $P_{\lambda_\nu}$ sufficiently large, we obtain $\xi_{\lambda_\nu\ell}$ such that
\begin{align}
        \sum_{\ell=1}^{A}a_\ell\xi_{\lambda_\nu\ell}&=0,
        \label{eq:xi-relation}\\
        P_{\lambda_\nu}^{(2)}&\longrightarrow\PP^{A-1}_k,
        \quad z\longmapsto[\xi_{\lambda_\nu1}(z):\cdots:\xi_{\lambda_\nu A}(z)]
        \quad\text{is a closed immersion,}                     \label{eq:xi-immersion}
\end{align}
and the closed point image avoids $E_{\nu-1}$.  Thus all closed point images obtained in this way are pairwise distinct.

Put
\[
        S=\set{P_\lambda\mid\lambda\in\Lambda},
        \qquad
        Z=\coprod_{\lambda\in\Lambda}P_\lambda^{(2)}.
\]

\subsection{The shielded Segre block}

Apply Lemma~\ref{lem:binary} to the finite set $S$ and to the fixed sections $e_j$.  We obtain a finite set
\[
        \mathcal R\subset |X|\setminus S
\]
and nonzero binary values $\epsilon_{Qi}$ for $Q\in\mathcal R$.  Choose an ample line bundle $\Ample$ on the fixed projective closure $\barX$.  Apply Lemma~\ref{lem:shield} to this fixed $j:X\hookrightarrow\barX$, to $\Ample$, and to $S$, $Z$, $\mathcal R$, and these values.  We obtain an integer $D>0$ and sections
\[
        \overline h_0,\ldots,\overline h_m\in H^0(\barX,\overline{\MM}),
        \qquad
        \overline{\MM}=\Ample^{\otimes D}.
\]
Write
\[
        \MM=j^*\overline{\MM},
        \qquad
        h_i=\overline h_i|_X .
\]
Then
\[
        h_i|_Z=0,
        \qquad
        h_i(Q)=\epsilon_{Qi}\qquad(Q\in\mathcal R),
        \qquad
        \bigcap_{i=0}^{m}V_{\barX}(\overline h_i)=S.
\]
By Lemma~\ref{lem:binary}, for every connected component $C$ the sections
\[
        h_i e_j\in H^0(C,(\MM\otimes\EE)|_C)
\]
are $k_C$-linearly independent, where $k_C$ is the algebraic closure of $k$ in $k(C)$.

For $n\ge1$ set
\[
        \overline h_i^{(n)}=\overline h_i^n\in H^0(\barX,\overline{\MM}^{\otimes n}),
        \qquad
        h_i^{(n)}=h_i^n\in H^0(X,\MM^{\otimes n}),
\]
\[
        \overline{\LL}_n=\overline{\MM}^{\otimes n}\otimes\overline{\EE},
        \qquad
        \LL_n=j^*\overline{\LL}_n=\MM^{\otimes n}\otimes\EE,
\]
and
\[
        \overline w_{ij}^{(n)}=\overline h_i^{(n)}\overline e_j\in H^0(\barX,\overline{\LL}_n),
        \qquad
        w_{ij}^{(n)}=h_i^{(n)}e_j\in H^0(X,\LL_n).
\]
Since $\epsilon^n=\epsilon$ for $\epsilon\in\set{0,1}$, the same binary tests show that the sections $w_{ij}^{(n)}$ are $k_C$-linearly independent on every connected component $C$ for every $n\ge1$.  Moreover
\[
        \overline w_{ij}^{(n)}|_Z=0,
        \qquad
        \bigcap_{i,j}V_{\barX}(\overline w_{ij}^{(n)})=S,
\]
because the $\overline e_j$ have no common zero on $\barX$ and the $\overline h_i$ have common zero set $S$.

\begin{lemma}[Finite tests for componentwise nondegeneracy]\label{lem:T}
For every connected component $C$ of $X$, let $k_C$ be the algebraic closure of $k$ in $k(C)$.  There exists a finite reduced subscheme
\[
        T_C\subset C\setminus S
\]
such that $\mathcal R\cap C\subset T_C$,
\[
        \dim_{k_C} H^0(T_C,\OO_{T_C})\ge N+1,
\]
and, for every $n\ge1$, the restrictions
\[
        w_{ij}^{(n)}|_{T_C}
        \in H^0(T_C,\LL_n|_{T_C})
\]
are $k_C$-linearly independent.
\end{lemma}

\begin{proof}
Fix a connected component $C$ and give
\[
        T_C^{(0)}:=\mathcal R\cap C
\]
the reduced finite subscheme structure.  Since $\mathcal R\subset |X|\setminus S$, we have $T_C^{(0)}\subset C\setminus S$.  For every nonzero matrix
\[
        \gamma=(\gamma_{ij})\in k_C^{(m+1)(r+1)},
\]
Lemma~\ref{lem:binary} gives a point $Q_{C,\gamma}\in\mathcal R\cap C=T_C^{(0)}$ such that
\[
        \left(\sum_{i,j}\gamma_{ij}h_i e_j\right)(Q_{C,\gamma})\ne0.
\]
At every point of $\mathcal R$ the values $h_i$ are $0$ or $1$, so $h_i(Q)^n=h_i(Q)$ for all $n\ge1$.  Hence the same point detects the corresponding linear combination of the sections $w_{ij}^{(n)}=h_i^n e_j$ for every $n\ge1$.  Therefore the restrictions
\[
        w_{ij}^{(n)}|_{T_C^{(0)}}
\]
are $k_C$-linearly independent for every $n\ge1$.

If
\[
        \dim_{k_C}H^0(T_C^{(0)},\OO_{T_C^{(0)}})\ge N+1,
\]
set $T_C=T_C^{(0)}$.  Otherwise, add closed points of $C\setminus(S\cup\mathcal R)$, each time avoiding the finite set already chosen and using Lemma~\ref{lem:closedpoints} over the finite field $k_C$, until the resulting finite reduced subscheme $T_C$ satisfies
\[
        \dim_{k_C}H^0(T_C,\OO_{T_C})
        =\sum_{P\in T_C}[\kk(P):k_C]
        \ge N+1.
\]
The inclusion $\mathcal R\cap C\subset T_C$ holds by construction, and adding points cannot destroy the already established linear independence.
\end{proof}

Let
\[
        T=\coprod_C T_C.
\]
By Lemma~\ref{lem:interpolation-projective}, after replacing $n$ by a sufficiently large integer, the restriction map
\begin{equation}\label{eq:globalinterpolation}
        H^0(\barX,\overline{\LL}_n)
        \longrightarrow
        H^0(Z,\overline{\LL}_n|_Z)\oplus H^0(T,\overline{\LL}_n|_T)
\end{equation}
is surjective.  Fix such an $n$ for the remainder of the proof and write
\[
        \overline{\LL}=\overline{\LL}_n,
        \qquad
        \LL=\LL_n,
        \qquad
        \overline w_{ij}=\overline w_{ij}^{(n)},
        \qquad
        w_{ij}=w_{ij}^{(n)}.
\]

\section{Prescribing the anti-code block}\label{sec:anticode}

We now choose global sections
\[
        \overline u_1,\ldots,\overline u_A\in H^0(\barX,\overline{\LL})
\]
and write $u_\ell=\overline u_\ell|_X$.
There are two independent requirements.  On $Z$, the first jets force the hyperplane $\lambda$ to vanish on $P_\lambda^{(2)}$.  On $T$, the values complement the shielded block to force componentwise nondegeneracy.

\subsection{Restrictions on the assigned first-order neighbourhoods}

For each $\lambda\in\Lambda$, choose a trivialization of $\overline{\LL}$ over $P_\lambda^{(2)}$ and transport the already constructed elements $\xi_{\lambda\ell}\in J_\lambda$ to elements
\[
        \eta_{\lambda\ell}
        \in H^0(P_\lambda^{(2)},\overline{\LL}|_{P_\lambda^{(2)}})
        \qquad(1\le\ell\le A).
\]
Then \eqref{eq:xi-relation} and \eqref{eq:xi-immersion} become
\begin{align}
        \sum_{\ell=1}^{A}a_\ell\eta_{\lambda\ell}&=0,                 \label{eq:anti-relation}\\
        P_\lambda^{(2)}&\longrightarrow\PP^{A-1}_k,
        \quad z\longmapsto [\eta_{\lambda1}(z):\cdots:\eta_{\lambda A}(z)]
        \quad\text{is a closed immersion.}                       \label{eq:anti-local-immersion}
\end{align}
Because the closed point images for different $\lambda$ are pairwise distinct, the scheme-theoretic images of the summands $P_\lambda^{(2)}$ have pairwise disjoint supports in $\PP^{A-1}_k$.  Lemma~\ref{lem:finite-disjoint-closed-immersion} therefore shows that the disjoint union map
\[
        Z=\coprod_{\lambda\in\Lambda}P_\lambda^{(2)}
        \longrightarrow\PP^{A-1}_k
\]
defined by the $\eta_{\lambda\ell}$ is a closed immersion.

Equivalently, for $1\le\ell\le A$ we have prescribed a section
\[
        \eta_\ell=(\eta_{\lambda\ell})_{\lambda\in\Lambda}
        \in H^0(Z,\overline{\LL}|_Z)
\]
with the property that, on the summand $P_\lambda^{(2)}$,
\begin{equation}\label{eq:anti-Z-global}
        \sum_{\ell=1}^{A}a_\ell\eta_\ell=0.
\end{equation}

\subsection{Restrictions on the component tests}

For a connected component $C$, let $k_C$ be the algebraic closure of $k$ in $k(C)$ and set
\[
        V_C=H^0(T_C,\overline{\LL}|_{T_C}),
\]
viewed as a $k_C$-vector space.  Since $T_C$ is finite and reduced and $\overline{\LL}|_{T_C}$ is invertible,
\[
        \dim_{k_C}V_C=\dim_{k_C}H^0(T_C,\OO_{T_C}).
\]
By Lemma~\ref{lem:T}, the vectors
\[
        w_{ij}|_{T_C}\in V_C
\]
are $k_C$-linearly independent; their span has dimension $B$.  Since
\[
        \dim_{k_C}V_C\ge N+1=A+B,
\]
we may choose vectors
\[
        \upsilon_{1,C},\ldots,\upsilon_{A,C}\in V_C
\]
such that
\begin{equation}\label{eq:component-basis}
        \upsilon_{1,C},\ldots,\upsilon_{A,C},
        \quad
        w_{00}|_{T_C},w_{01}|_{T_C},\ldots,w_{mr}|_{T_C}
\end{equation}
are $k_C$-linearly independent in $V_C$.  Combining the components gives
\[
        \upsilon_\ell=(\upsilon_{\ell,C})_C
        \in H^0(T,\overline{\LL}|_T)
        \qquad(1\le\ell\le A).
\]

By the surjectivity of \eqref{eq:globalinterpolation}, choose global sections
\[
        \overline u_1,\ldots,\overline u_A\in H^0(\barX,\overline{\LL})
\]
satisfying
\begin{align}
        \overline u_\ell|_Z&=\eta_\ell,                                  \label{eq:u-on-Z}\\
        \overline u_\ell|_T&=\upsilon_\ell.                              \label{eq:u-on-T}
\end{align}
Set $u_\ell=\overline u_\ell|_X$.

\section{The embedding and the hyperplane singularities}\label{sec:embedding}

Consider the ordered collection of sections
\[
        \overline u_1,\ldots,\overline u_A,
        \overline w_{00},\ldots,\overline w_{mr}
        \in H^0(\barX,\overline{\LL}).
\]

\begin{lemma}[No base points]\label{lem:basepoint}
The sections
\[
        \overline u_1,\ldots,\overline u_A,
        \overline w_{00},\ldots,\overline w_{mr}
\]
have no common zero on $\barX$.  Consequently their restrictions
\[
        u_1,\ldots,u_A,
        w_{00},\ldots,w_{mr}
\]
have no common zero on $X$.
\end{lemma}

\begin{proof}
If $x\notin S$, then not all $\overline h_i(x)$ vanish and not all $\overline e_j(x)$ vanish, so some product
\[
        \overline w_{ij}(x)=\overline h_i(x)^n\overline e_j(x)
\]
is nonzero.  If $x=P_\lambda\in S$, then \eqref{eq:anti-local-immersion} implies that the anti-code values
\[
        \eta_{\lambda1},\ldots,\eta_{\lambda A}
\]
are not all zero at $P_\lambda$; by \eqref{eq:u-on-Z}, the same is true for the $\overline u_\ell$.
\end{proof}

By Lemma~\ref{lem:basepoint}, the ordered collection defines a morphism
\[
        \overline\iota_N:\barX\longrightarrow\PP^N_k.
\]
Let
\[
        \iota_N=\overline\iota_N|_X:X\longrightarrow\PP^N_k.
\]
Thus, for $x\in X$,
\begin{equation}\label{eq:iota}
        \iota_N(x)=
        [u_1(x):\cdots:u_A(x):w_{00}(x):w_{01}(x):\cdots:w_{mr}(x)].
\end{equation}

\begin{lemma}[Local immersion and point separation]\label{lem:embedding}
The morphism $\iota_N$ is a locally closed embedding.
\end{lemma}

\begin{proof}
On $\barX\setminus S$, the $\overline w$-block has no common zero and is the Segre product of
\[
        \overline h^{(n)}:\barX\setminus S\longrightarrow\PP^m_k,
        \qquad
        x\longmapsto[\overline h_0(x)^n:\cdots:\overline h_m(x)^n]
\]
and the fixed closed immersion
\[
        \overline e:\barX\longrightarrow\PP^r_k.
\]
Thus
\[
        \barX\setminus S\longrightarrow\PP^{B-1}_k,
        \qquad
        x\longmapsto[\overline w_{00}(x):\cdots:\overline w_{mr}(x)]
\]
is the composition
\[
        \barX\setminus S\xrightarrow{(\overline h^{(n)},\overline e)}\PP^m_k\times\PP^r_k
        \xrightarrow{\operatorname{Segre}}\PP^{B-1}_k.
\]
The second projection of $(\overline h^{(n)},\overline e)$ is the locally closed embedding $\overline e|_{\barX\setminus S}$, so $(\overline h^{(n)},\overline e)$ is a locally closed embedding; the Segre map is a closed embedding.  Hence the $\overline w$-block defines a locally closed embedding of $\barX\setminus S$.

On $Z$, all $\overline w_{ij}$ vanish and the $\overline u$-block defines a closed immersion $Z\hookrightarrow\PP^{A-1}_k$ by construction.  Since $S\subset X$ and $X$ is a smooth open subscheme of $\barX$, the set $S$ is contained in the smooth locus of $\barX$.  Lemma~\ref{lem:basepoint} says that the full collection of sections has no common zero on $\barX$.  Therefore the hypotheses of Lemma~\ref{lem:two-block-separation} are satisfied for $Y=\barX$, with the $B$ sections $w_1,\ldots,w_B$ enumerated as the $\overline w_{ij}$.  Hence
\[
        \overline\iota_N:\barX\hookrightarrow\PP^N_k
\]
is a closed immersion.  Since $X\subset\barX$ is open, the restriction $\iota_N=\overline\iota_N|_X$ is a locally closed embedding.
\end{proof}

\begin{lemma}[Componentwise nondegeneracy]\label{lem:nondegenerate}
For every connected component $C$ of $X$, every field extension $K/k$, and every connected component $D$ of $C_K$, the image of $D$ under $(\iota_N)_K$ is not contained in any hyperplane of $\PP^N_K$.
\end{lemma}

\begin{proof}
Let $C$ be a connected component and let $k_C$ be the algebraic closure of $k$ in $k(C)$.  Suppose that a $k_C$-linear combination
\[
        s=\sum_{\ell=1}^{A}\alpha_\ell u_\ell+
        \sum_{i=0}^{m}\sum_{j=0}^{r}\beta_{ij}w_{ij}
        \in H^0(C,\LL|_C)
\]
vanishes identically on $C$.  Restricting to $T_C$ gives
\[
        \sum_{\ell=1}^{A}\alpha_\ell \upsilon_{\ell,C}
        +\sum_{i,j}\beta_{ij}w_{ij}|_{T_C}=0
        \quad\text{in}\quad H^0(T_C,\LL|_{T_C}),
\]
where the target is viewed as a $k_C$-vector space.  By \eqref{eq:u-on-T}, $u_\ell|_{T_C}=\upsilon_{\ell,C}$ for every $\ell$, and the construction preceding \eqref{eq:component-basis} chooses the vectors so that
\[
        \upsilon_{1,C},\ldots,\upsilon_{A,C},
        \quad
        w_{00}|_{T_C},w_{01}|_{T_C},\ldots,w_{mr}|_{T_C}
\]
are $k_C$-linearly independent in $H^0(T_C,\LL|_{T_C})$.  This is exactly \eqref{eq:component-basis}.  Hence every coefficient $\alpha_\ell$ and $\beta_{ij}$ is zero.  The coordinate sections are therefore $k_C$-linearly independent on $C$.

By Lemma~\ref{lem:basepoint}, these $N+1$ coordinate sections are basepoint-free on $C$.  Lemma~\ref{lem:component-base-change}, applied to them, gives the asserted linear nondegeneracy on every connected component after every scalar extension $K/k$.
\end{proof}

We now prove that every rational hyperplane section is singular.  Let
\[
        \lambda=[a:b]\in\Lambda=(\PP^N_k)^\vee(k).
\]
The pullback of $H_\lambda$ is the section
\[
        s_\lambda=
        \sum_{\ell=1}^{A}a_\ell u_\ell+
        \sum_{i=0}^{m}\sum_{j=0}^{r}b_{ij}w_{ij}
        \in H^0(X,\LL).
\]
At the assigned point $P_\lambda$, every $w_{ij}$ vanishes on $P_\lambda^{(2)}$ because every $h_i$ vanishes on $Z$.  By \eqref{eq:anti-Z-global} and \eqref{eq:u-on-Z},
\[
        \left.\sum_{\ell=1}^{A}a_\ell u_\ell\right|_{P_\lambda^{(2)}}=0.
\]
Therefore
\begin{equation}\label{eq:singular-jet}
        s_\lambda|_{P_\lambda^{(2)}}=0.
\end{equation}
Let $C_\lambda$ be the connected component of $X$ containing $P_\lambda$.  By Lemma~\ref{lem:nondegenerate}, the image of $C_\lambda$ is not contained in the rational hyperplane $H_\lambda$, so $s_\lambda|_{C_\lambda}$ is not identically zero.  Since $C_\lambda$ is integral, the germ of $s_\lambda$ at $P_\lambda$ is nonzero: if the germ were zero, then $s_\lambda$ would vanish on a nonempty open neighbourhood of $P_\lambda$, hence at the generic point of $C_\lambda$, and therefore identically on $C_\lambda$.

Since $\iota_N$ is a locally closed embedding, the scheme-theoretic intersection $H_\lambda\cap\iota_N(X)$ is identified near $P_\lambda$ with the zero scheme of $s_\lambda$ on $X$.  Trivialize $\LL$ near $P_\lambda$ and let
\[
        0\ne f_\lambda\in R:=\OO_{X,P_\lambda}
\]
represent $s_\lambda$.  Equation \eqref{eq:singular-jet} says
\[
        f_\lambda\in\mm_{P_\lambda}^2.
\]
The local ring of the hyperplane section at $P_\lambda$ is
\[
        R/(f_\lambda).
\]
The ring $R$ is a regular local domain of dimension $m$, because $X$ is smooth of pure dimension $m$ over $k$.  Lemma~\ref{lem:first-order-singular} applies to the nonzero element $f_\lambda\in\mm_{P_\lambda}^2$ and shows that $R/(f_\lambda)$ is not regular.  Since every smooth scheme over a field is regular, the hyperplane section is not smooth at the image point $\iota_N(P_\lambda)$; indeed it is singular there.

This proves Theorem~\ref{thm:main} with
\[
        N_X=B+m+2=(m+1)(r+1)+m+2.
\]

\section{Consequences and comparison}\label{sec:comparison}

\begin{corollary}[Negative answer to Baker's question for a prescribed variety]\label{cor:poonen}
Let $X$ be nonempty smooth quasiprojective of pure positive dimension over $\F_q$.  There is no integer $n_0$ such that every embedding
\[
        \iota:X\hookrightarrow\PP^n_k,
        \qquad n\ge n_0,
\]
satisfying Poonen's componentwise nondegeneracy hypothesis admits a $k$-rational hyperplane $H\subset\PP^n_k$ for which $H\cap\iota(X)$ is smooth of dimension $\dim X-1$.
\end{corollary}

\begin{proof}
For every $N\ge N_X$, Theorem~\ref{thm:main} gives an embedding $\iota_N$ satisfying the nondegeneracy hypothesis, even after arbitrary scalar extension, but for which every $k$-rational hyperplane section is singular at its assigned point.
\end{proof}

\begin{remark}[Scope of the conclusion]
Theorem~\ref{thm:main} should be read as a construction of counterexample embeddings for the prescribed variety $X$.  It does not classify all high-dimensional embeddings of $X$.  The point is that, after $X$ is fixed, the construction supplies embeddings satisfying Poonen's componentwise nondegeneracy hypothesis but having no smooth $k$-rational hyperplane section.  In fact this can be done for every $N\ge N_X$.

Every rational hyperplane is assigned a specific singular point at which the pulled-back equation has zero first jet:
\[
        \forall\lambda\in(\PP^N_k)^\vee(k)
        \quad
        \exists P_\lambda\in X
        \quad
        s_\lambda\in\mm_{P_\lambda}^2\LL_{P_\lambda}.
\]
The construction does not assert that the embeddings are given by complete linear series.  This is consistent with Erman--Wood's complete-linear-series construction for special projective examples and with their observation that a construction for a prescribed $X$ should not generally use complete linear series \cite[Remark 9.4]{ErmanWood2015}.
\end{remark}

\begin{remark}[Role of the two blocks]
The finite-field obstruction is not merely that the dual projective space has finitely many $k$-points.  For each
\[
        \lambda=[a:b]\in(\PP^N_k)^\vee(k)
\]
one must impose the zero-first-jet condition
\[
        \sum_\ell a_\ell u_\ell+
        \sum_{i,j}b_{ij}h_i^n e_j
        \equiv0\pmod{\mm_{P_\lambda}^2}
\]
while keeping the same coordinate system basepoint-free, immersive at $S$, immersive away from $S$, point-separating, and nondegenerate on each component.  The shielded block makes all $b$-terms vanish modulo $\mm_{P_\lambda}^2$ at the assigned points, while the anti-code block absorbs the one remaining $a$-relation without losing the first-order embedding of $P_\lambda^{(2)}$.  The auxiliary finite sets $\mathcal R$ and $T_C$ prevent these local singularity requirements from collapsing the global image into a rational hyperplane.
\end{remark}

\begin{acknowledgements}
The author used GPT-based large language models to generate and compare possible proof strategies, suggest auxiliary lemmas, identify potential gaps, and assist with preliminary wording. All AI-generated suggestions were manually reviewed, selected, modified, and independently verified by the author. No unverified AI-generated proof or mathematical claim was incorporated into the final manuscript.

The author also used the AI agents Rethlas (available at
\url{https://github.com/frenzymath/Rethlas}) as auxiliary proof-checking tools to inspect parts of the argument for possible logical gaps, missing assumptions, notation inconsistencies, and unclear inferential steps. Their outputs were used only as advisory feedback.The authors used the CSSC framework, available at
\url{https://github.com/anetigone/cssc}, together with GPT-based assistance, to
develop Lean 4 formalizations of selected proof components in this paper. The
resulting Lean 4 development is available at
\url{https://github.com/anetigone/lean-research-formalizations}. The
formalization is intended as a machine-checkable supplement to the mathematical
arguments presented here, and should be understood within the scope described in the accompanying repository.

The author retains full responsibility for the final content, including all definitions, statements, proofs, citations, and conclusions. The AI systems and proof-checking agents were not treated as authors and bear no responsibility for the manuscript.

\textit{Competing interests.} The author declares no competing interests.

\textit{Data access statement.} No datasets were generated or analysed during the current study.
\end{acknowledgements}

\affiliationone{%
   Yutong Zhang\\
   Yaoran Yang\\
   School of Mathematics, Sichuan University\\
   Chengdu, 610065\\
   China
   \email{yutongzhang@stu.scu.edu.cn}
   \email{yangyaoran@stu.scu.edu.cn}}

\end{document}